\documentclass[12pt]{article}
\usepackage{amsmath, amsthm, amssymb}
\usepackage[pdftex]{graphicx}
\usepackage{enumerate}
\usepackage{mathrsfs}
\usepackage{color}
\def\XXint#1#2#3{{\setbox0=\hbox{$#1{#2#3}{\int}$}
    \vcenter{\hbox{$#2#3$}}\kern-.5\wd0}}

\def\({\left(}
\def\){\right)}

\def\RR{{\mathbb{R}}}
\def\CC{{\mathbb{C}}}

\newcommand{\tta}{\theta}

\newcommand{\dis}{\displaystyle}

\newcommand{\be}{\begin{equation}}
\newcommand{\ee}{\end{equation}}
\newcommand{\bea}{\begin{eqnarray}}
\newcommand{\eea}{\end{eqnarray}}
\newcommand{\beann}{\begin{eqnarray*}}
\newcommand{\eeann}{\end{eqnarray*}}
\newcommand{\nnn}{\nonumber}

 \newtheorem{theorem}{Theorem}[section]

\newtheorem{remark}[theorem]{Remark}

\newtheorem{lemma}[theorem]{Lemma}

\newtheorem{prop}[theorem]{Proposition}

\begin{document}

\title{Stability of symmetric vortices for two-component Ginzburg--Landau systems}
\author{{\Large Stan Alama}\footnote{Dept. of Mathematics and Statistics,
McMaster Univ., Hamilton, Ontario, Canada L8S 4K1.  Supported
by an NSERC Discovery Grant.} \and  {\Large Qi Gao${}^*$}
 }

\thispagestyle{empty}
\maketitle

\begin{abstract}
We study Ginzburg--Landau equations for a complex vector order parameter
$\Psi=(\psi_+,\psi_-)\in\mathbb C^2$.  We consider the Dirichlet problem in the disk in $\mathbb{R}^2$ with a symmetric,
degree-one boundary condition, and study its stability, in the sense of the spectrum of the second variation of the energy.  We find that the stability of the degree-one equivariant solution depends on both the Ginzburg-Landau parameter as well as the sign of the interaction term in the energy.

\bigskip


\medskip


\end{abstract}

\newpage


\baselineskip=18pt

\section{Introduction}
Let $\Omega\subset\mathbb{R}^2$ be a smooth, bounded domain, and $\Psi\in H^{1}(\Omega;\mathbb{C}^2 )$. We define an energy functional
 \begin{equation}\label{energy}
E_{\lambda}(\Psi; \Omega)=\\
\dis\int_{\Omega}\frac{1}{2}|\nabla\Psi|^{2}
+\frac{\lambda}{4}[A_{+}(|\psi_{+}|^{2}-t^2_{+})^{2}+A_{-}(|\psi_{-}|^{2}-t^2_{-})^{2}+2B(|\psi_{+}|^{2}-t^2_{+})(|\psi_{-}|^{2}-t^2_{-})]\ ,
\end{equation}
where $\Psi=[\psi_+ (x), \psi_- (x)]\in\mathbb{C}^2$, $A_\pm >0$, $B$ and $\lambda>0$ are parameters. Energy functionals 
of a form similar to $E_\lambda$ have been introduced in physical models of Bose-Einstein condensates \cite{ktu03} and for ruthenate superconductors \cite{kr98}; we will discuss these applications briefly at the end of the introduction. 

Throughout the paper we make the following assumptions concerning the constants appearing in \eqref{energy}:
\begin{equation*}  
A_+, A_->0, \ B^2<A_+A_-, \quad t_+, t_->0.
 \tag{\text{H} } 
 \end{equation*}
 By hypothesis (H), the potential term in the energy
$$  F(\Psi) = A_{+}(|\psi_{+}|^{2}-t^2_{+})^{2}+A_{-}(|\psi_{-}|^{2}-t^2_{-})^{2}+2B(|\psi_{+}|^{2}-t^2_{+})(|\psi_{-}|^{2}-t^2_{-})  $$ 
is nonnegative and coercive, and attains its minimum (of zero) when $|\psi_\pm|=t_\pm$. 
 As $\lambda\to\infty$, minimizers $\Psi$ should lie on the manifold in $\mathbb{C}^2$
on which the potential $F(\Psi)$ vanishes. That manifold is a 2-torus $\mathbb{T}\subset\mathbb{C}^2$, parameterized by two
real phases $\Psi=[\psi_+,\psi_- ]=[t_+ e^{i\alpha_+}, t_- e^{i\alpha_-}]$, and thus a $\mathbb{T}^2$-valued map $\Psi(x)$ carries a pair of integer-valued 
degrees around any closed curve $C$,
\begin{equation*}
\text{deg}(\Psi; C)=[N_+, N_- ],\quad N_+ =\text{deg}(\psi_+; C),\quad N_- =\text{deg}(\psi_-; C).
\end{equation*}
If along the boundary $\partial\Omega$, $\Psi$ has nonzero degree in either component, then there is no finite energy map $\Psi$ which takes
values in $\mathbb{T}$ and satisfies those boundary conditions, and we expect that {\it vortices} of solutions will be created in the $\lambda\to\infty$, just as in the classical
Ginzburg-Landau model \cite{bbh94book}.

In this paper we will restrict our attention to the unit disk $\Omega=\mathbb{D}_R$, with symmetric Dirichlet boundary conditions 
$\psi_\pm \big|_{\partial\mathbb{D}_R} =t_\pm e^{i N_\pm\theta}$, with $N_\pm\in\mathbb Z$.  For these special boundary conditions, there exist equivariant solutions (written in polar coordinates,) $\Psi(x)=(f_+(r)e^{iN_+\theta}, f_-(r)e^{iN_-\theta})$.  

For  the classical (single-component) Ginzburg-Landau equations in $\mathbb{R}^2$,
 \begin{equation}\label{GL}
 \left\{ 
 \begin{array}{ll}
\displaystyle -\Delta u=\lambda u(1-|u|^2 )u\ ,&\text{in}\ \ \mathbb{D}_1,\\
 u=z^d\ ,&\text{on}\ \ \partial\mathbb{D}_1,
 \end{array}
 \right.
 \end{equation}
much is known about these solutions. For any degree $d$, there exists a unique equivariant solution, of the form $u=f(r)e^{id\theta}$ in \cite{hh94}, 
with $f'(r)>0$.
As noted in \cite{bbh94book}, if $\lambda$ is smaller than the first Dirichlet eigenvalue in $\Omega$, then this is the unique solution to \eqref{GL}. 
 Bauman-Carlson-Phillips \cite{bcp93} proved that (in any domain $\Omega$) a local minimizer with degree 1 vanishes at a single point. 
 Mironescu \cite{m95} discussed the stability 
 of the radial solution to \eqref{GL} as a function of the given degree $d$, showing that the equivariant solution with degree one is stable for any $\lambda>0$, while for $|d|\ge 2$, the solution becomes unstable at a critical value of $\lambda$. Comte and Mironescu \cite{cm98} proved that 
the loss of stability (see \cite{m95}) leads to the appearance of a bifurcation from the branch of equivariant solutions for $d=2$ (see also \cite{sauv03}).

Given the results on the single-component Ginzburg--Landau model, it is natural to specialize to the degree-1 case.  Indeed, in \cite{thesis} it is shown that {\em entire} solutions of the Euler-Lagrange equations (in all $\RR^2$) of the two-component system \eqref{energy} with degree (at infinity) $|n_\pm|\ge 2$ are unstable.  
In this paper we study the stability of the equivariant solutions with degree $[N_+,N_-]=[1,1]$ of the Dirichlet problem in the unit disk $\mathbb{D}_1$. 
By stability, we mean positivity of the second variation of the energy $E_\lambda$ as a quadratic form on $H^1_0(\mathbb D_1;\CC^2)$.  Dynamical notions of stability depend on the type of evolution equation chosen for the model, and for the natural Schr\"odinger dynamics for applications to Bose-Einstein condensates 
there are many different definitions of stability (see \cite{pe11}), most of which are in some way related to the second variation of the energy

 We write the energy functional in the unit disk as
\begin{multline}\label{energylam}
E_{\lambda}(\Psi)=\dis \int_{\mathbb{D}_{1}}{1\over 2}|\nabla\Psi|^2 +{\lambda\over4}\left[A_{+}(|\psi_{+}|^{2}-t^2_{+})^2\right.\\
                                 \dis \left.+A_{-}(|\psi_{-}|^{2}-t^2_{-})^2  +2B(|\psi_{+}|^{2}-t^2_{+})(|\psi_{-}|^{2}-t^2_{-})\right]\ ,
\end{multline}
and consider minima (or more generally, critical points) of $E_\lambda (\Psi)$ over the space {\bf H}, consisting
of all functions $\Psi\in H^{1}(\mathbb{D}_1; \mathbb{C}^2 )$ with the symmetric boundary condition:
\begin{equation}\label{boundary}
\Psi\big|_{\partial\mathbb{D}_1}=[t_+ e^{i\tta}, t_- e^{i\tta}].
\end{equation}
The associated Euler-Lagrange equations to $E_\lambda (\Psi)$ with boundary condition \eqref{boundary} is as follows:
\begin{equation}\label{dirichlet}
\left\{
\begin{array}{ll}
-\Delta\psi_{+}+\lambda[A_{+}(|\psi_+ |^{2}-t^2_+ )+B(|\psi_- |^{2}-t^2_- )]\psi_{+}=0, & \text{in}\ \mathbb{D}_1,\\
-\Delta\psi_{+}+\lambda[A_{-}(|\psi_- |^{2}-t^2_- )+B(|\psi_+ |^{2}-t^2_+ )]\psi_{-}=0, & \text{in}\ \mathbb{D}_1,\\
\psi_{\pm}=t_\pm e^{i\theta}, & \text{on}\ \partial\mathbb{D}_1 .
\end{array}
\right.
\end{equation}
We note that rescaling by $R=\sqrt{\lambda}$, this Dirichlet problem with $\lambda=1$ is equivalent to the Dirichlet problem in a very large
disk $\mathbb{D}_R$, as $R\to\infty$ when $\lambda\to\infty$. Thus, this problem is also an approximation to the stability problem for
entire solutions in all $\mathbb{R}^2$.

The two-component model \eqref{energy} was studied by \cite{abm09}, \cite{abm13}, in the ``balanced" case, $A_\pm =1$, $t^2_\pm ={1\over2}$, for which the two components $\psi_\pm$ are interchangeable.
Analogous to the arguments in \cite{abm13}, we note that the stability of degree $[1,1]$ vortex to
\eqref{dirichlet} depends on both the Ginzburg-Landau parameter $\lambda$ and the sign of the interaction coefficient $B$ in the energy.  The main theorem is as follows:
\begin{theorem}\label{main}
Let $A_+, A_->0$ be fixed, and $B$ such that $B^2<A_+A_-$.
Assume $\Psi(x; \lambda)=[f_{+}(r; \lambda)e^{i\tta},f_{-}(r; \lambda)e^{i\tta}]$ is the equivariant solution for those parameters $A_\pm,B$
to \eqref{dirichlet}.
\begin{description}
  \item[(i)] If $B<0$, then $\Psi(x; \lambda)$ is stable for $\forall \lambda>0$, in the sense $E''_{\lambda}(\Psi)[\Phi]>0$
   $\forall\Psi\in H^1_{0}(\mathbb{D}_{1};\mathbb{C}^2 )$.
  \item[(ii)] There is $B_0 >0$ such that for any $B\in (0, B_0 )$, there exists a unique constant $\lambda_{*}=\lambda_{*}(B)>0$ such that $\Psi(x;\lambda)$ is unstable for any $\lambda>\lambda_{*}$.
\end{description}
\end{theorem}

The restriction $0<B<B_0$ is imposed to guarantee that the entire solution (in all $\RR^2$) be monotone increasing in each component.  The existence of such a $B_0$ follows from \cite{ag13}.  Indeed, it may be deduced from asymptotic expansions done in \cite{ag13} that the monotonicity of both $f_\pm(r)$ cannot hold for all $0<B<\sqrt{A_+A_-}$, as there is a critical value of $B$ for which one component must approach its limiting value $t_\pm$ from above as $r\to\infty$.  Nevertheless, the exact value of $B_0$ is not explicitly known.

Theorem~\ref{main} generalizes results of \cite{abm13}, which were restricted to the ``balanced" case,  $A_\pm =1$, $t^2_\pm ={1\over2}$, and $-1<B<1$.  The balanced case presents many simplifications:  in particular, for any $B$, the equivariant solution has the form 
$\Psi = {1\over\sqrt{2}}(u,u)$, where $u=f(r)e^{i\theta}$ is the solution to the classical Ginzburg-Landau problem, and the solution space is invariant under an involution, $(\psi_+(x),\psi_-(x))\mapsto (-\psi_-(-x),-\psi_+(-x))$.  This symmetry enables the authors to reduce the problem of stability from a system to a single complex-valued equation. The results of \cite{abm13} are somewhat sharper, and they prove that when $B<0$, the radial solution 
is {\it unique}. In particular, they show that when $B>0$, a vortex of degree $[n_+, n_- ]=[1,1]$ is  {\it not} radially symmetric: it must have non-coincident zeros in its two components, $\psi_\pm$.

Following \cite{abm13}, it is natural to suppose that bifurcation occurs at $\lambda_* (B)$, when $B\in (0, B_0 )$. As in \cite{abm13}, the unstable
directions correspond to separating the single $[1,1]$ vortex at the origin in two antipidal vortices with degree $[n_+, n_- ]=[1,0]$ and $[0,1]$.
Formally, the second variation of energy on the plane $\RR^2$ is negative along the direction $\Phi=(\varphi_+,\varphi_-)$ with
\begin{equation}\label{split}
\varphi_\pm 
   = \pm{\partial\over\partial x_1} \left(f_\pm(r) e^{i\theta}\right)
    = \pm\frac12\left( f'_\pm(r) +{1\over r}f_\pm(r)  \right)   
     \pm\frac12\left(f'_\pm(r) -{1\over r}f_\pm(r)  \right)e^{2i\theta}.
\end{equation}
We conjecture that the same separation phenomenon holds in the more general case as well, although the methods of \cite{abm13} do not carry over to the more general case in a straightforward way.

The plan of the paper is as follows:  section 2 reviews some relevant results on the equivariant solutions and their properties.  Section 3 covers the decomposition of the second variation in Fourier modes, and identifies the two key terms in the decomposition.
The proof of part (i) of Theorem~\ref{main} is given in Section~4, following the strategy
in \cite{m95}, \cite{abm13}.  Section 5 is devoted to the instability result stated in part (ii) of Theorem~\ref{main}.

\subsection*{Two-component Bose-Einstein condensates}

In a model of a two-compenent BEC (see \cite{eto11},) we consider a pair of complex wave functions $\Phi=(\varphi_1,\varphi_2)\in\CC^2$, defined in the sample domain $\Omega\subset\RR^2$.  The energy of the configuration is defined as
$$  E(\varphi_1, \varphi_2) = 
  \int_\Omega\left[ {\hbar^2\over 2m_1}|\nabla\varphi_1|^2
     + {\hbar^2\over 2m_2}|\nabla\varphi_2|^2
      + \frac12\left( g_1|\varphi_1|^4 + g_2|\varphi_2|^4
            + 2g_{12}|\varphi_1|^2|\varphi_2|^2\right)\right] dx,
$$
where $m_1,m_2>0$ are the masses, and the coupling constants satisfy the positivity condition $g_1g_2-g_{12}^2>0$.  The Gross-Pitaevskii equations govern the dynamics of the condensate,
\begin{equation}\label{GP} \left.
\begin{aligned}
i\hbar\partial_t\varphi_1 &= -{\hbar^2\over 2m_1}\Delta\varphi_1
       + g_1|\varphi_1|^2\varphi_1 + g_{12}|\varphi_2|^2\varphi_1, \\
i\hbar\partial_t\varphi_2 &= -{\hbar^2\over 2m_2}\Delta\varphi_2
       + g_{12}|\varphi_1|^2\varphi_2 + g_{2}|\varphi_2|^2\varphi_2.
\end{aligned}
\right\}
\end{equation}
A stationary equation of the desired form is obtained by considering standing wave solutions, $\varphi_i(x,t)= e^{-i\mu_i t/\hbar}u_i(x)$, $i=1,2$, where $\mu_i$ represent the chemical potentials:
\begin{align*}
 -{\hbar^2\over 2m_1}\Delta\varphi_1
       + g_1|\varphi_1|^2\varphi_1 + g_{12}|\varphi_2|^2\varphi_1
       &= \mu_1\varphi_1, \\
 -{\hbar^2\over 2m_2}\Delta\varphi_2
       + g_{12}|\varphi_1|^2\varphi_2 + g_{2}|\varphi_2|^2\varphi_2
        &=  \mu_2\varphi_2.
\end{align*}
In the variational formulation of the stationary problem, the chemical potentials represent Lagrange multipliers, which arise because of the constraints on the masses of the two condensate species, 
$$  \int_\Omega |\varphi_i|^2 dx = \int_\Omega |u_i|^2 dx = N_i, \quad i=1,2.  $$
By a rescaling of the dependent variables, $\psi_+=\sqrt[4]{{m_2\over m_1}}\,u_1$, $\psi_-=\sqrt[4]{{m_1\over m_2}}\,u_2$, we may eliminate the masses $m_i$ from the equations, and we obtain the system 
\begin{equation}\label{eqns}
\left\{\begin{array}{l}
-\Delta\psi_{+}+\lambda [A_{+}(|\psi_{+}|^{2}-t^2_{+})+B(|\psi_{-}|^{2}-t^2_{-})]\psi_{+}=0,\\ 
-\Delta\psi_{-}+\lambda [A_{-}(|\psi_{-}|^{2}-t^2_{-})+B(|\psi_{+}|^{2}-t^2_{+})]\psi_{-}=0,
\end{array}\right.
\end{equation}
with $\lambda={\sqrt{m_1m_2}\over \hbar^2}$, $A_+={m_1\over m_2}g_1$, $A_-={m_2\over m_1}g_2$, $B=g_{12}$, and 
$$ t_+^2= {\mu_1 g_2-\mu_2 g_{12}\over g_1g_2-g_{12}^2}     \sqrt{{m_2\over m_1}}, \qquad
     t_-^2= 
     {\mu_2 g_1-\mu_1 g_{12}\over g_1g_2-g_{12}^2}\sqrt{{m_1\over m_2}} .
$$
A more realistic model of a BEC would include a term representing uniform rotation, and a trapping potential to replace the unphysical Dirichlet boundary condition \eqref{boundary}.  Nevertheless, it is to be expected that minimizers $\Psi_\lambda$ to the true functional for a BEC will resemble minimizing solutions of the Dirichlet problem \eqref{dirichlet} when restricted to a neighborhood of an isolated vortex (see \cite{bbh94book}.)

\section{Properties of equivariant solutions}
In this section we recall some essential properties of equivariant solutions of the Ginzburg-Landau system, proven in our prevous paper \cite{ag13}. We define the inner product as follows: $\langle u,v\rangle=\text{Re}(\bar{u}v)={{\bar{u}v+u\bar{v}}\over 2}$. Recall that in rescaling by $R=\sqrt{\lambda}$, the Dirichlet problem \eqref{dirichlet} in the unit disk is equivalent to the Dirichlet problem in the
disk $\mathbb{D}_R$ for the system,
\begin{equation}\label{Reqns}
\left.
\begin{gathered}
-\Delta\psi_{+}+[A_{+}(|\psi_{+}|^{2}-t^2_{+})+B(|\psi_{-}|^{2}-t^2_{-})]\psi_{+}=0,\\ 
-\Delta\psi_{-}+[A_{-}(|\psi_{-}|^{2}-t^2_{-})+B(|\psi_{+}|^{2}-t^2_{+})]\psi_{-}=0,
\end{gathered}
\right\}
\end{equation}
with symmetric degree-one boundary condition 
$\Psi|_{\partial\mathbb{D}_R}=(t_+ e^{i\theta}, t_- e^{i\theta})$.
Equivariant solutions to \eqref{Reqns} satisfy $\Psi=[\psi_+, \psi_- ]=[f_+ (r)e^{i\theta}, f_- (r)e^{i\theta}]$.  Under this ansatz the system
\eqref{Reqns} reduces to a system of ODEs with boundary condition,
\be\label{radeq}
\left.
\begin{gathered}
 \dis -f''_{+}-\frac{1}{r}f'_{+}+\frac{1}{r^2}f_{+}+\left[ A_{+}(f^2_{+}-t^2_{+})+B(f^2_{-}-t^2_{-})\right]f_{+}=0, \\
 \dis -f''_{-}-\frac{1}{r}f'_{-}+\frac{1}{r^2}f_{-}+\left[ A_{-}(f^2_{-}-t^2_{-})+B(f^2_{+}-t^2_{+})\right]f_{-}=0, \\
 f_\pm(R)=t_\pm.
\end{gathered}
\right\}
\ee
We will also consider solutions
$f^\infty_\pm(r)$ defined for all $r\in (0,\infty)$, with asymptotic conditions $f^\infty_\pm(r)\to t_\pm$ as $r\to\infty$, corresponding to entire equivariant solutions to \eqref{Reqns}.
Since $f_\pm(r)=|\psi_\pm|$, we require all solutions to satisfy $f_\pm(r)\ge 0$ for all $r$. 

We begin with an {\it a priori} bound on solutions of the system \eqref{Reqns} in bounded domains.  Let $\lambda_{s}>0$ denote the smallest eigenvalue of the matrix $\left[\begin{array}{cc}A_{+} & B \\B & A_{-}\end{array}\right]$.  The following lemma is proven in \cite{ag13}:
\begin{lemma}\label{uniformbound}
Let $\Omega\subset\RR^2$ be a bounded smooth domain, and 
$\Psi=(\psi_+,\psi_-)$ be a solution of \eqref{Reqns} in $\Omega$, satisfying
\begin{equation}\label{BC}
  |\Psi(x)|^2 = |\psi_+(x)|^2 + |\psi_-(x)|^2 \le t_+^2 + t_-^2 \quad
\text{for all $x\in\partial\Omega$}.
\end{equation}
Then $|\Psi(x)|\le\Lambda$, where $\Lambda^2 :=\min\{{2M\over \lambda_{s}}, t^2_{+}+t^2_{-} \}$ for all $x\in\Omega$, with $M=\max\{A_{+}t^2_{+}+Bt^2_{-}, A_{-}t^2_{-}+Bt^2_{+}\}$.
 \end{lemma}
 
The existence and uniqueness of equivariant solutions to \eqref{Reqns} in any ball $\mathbb{D}_R$ with boundary condition follows from standard methods:

\begin{prop}\label{existR}
For any $R>0$ there exists a unique solution $(f_+(r),f_-(r))$ of \eqref{radeq} with $f_\pm(r)\ge 0$ for all $r\in (0,R)$.
\end{prop}

The existence statement follows from the direct method in the calculus of variations; uniqueness may be proven using the method of Brezis \& Oswald \cite{bo86} (see \cite{ag13} for details).
  In \cite{ag13} we also showed the existence of unique {\em entire} equivariant solutions, with the following properties:

\begin{prop}\label{lemunique}
Let  $A_{+}A_{-}-B^{2}>0$. Then there exists a unique solution $[f^\infty_+(r),\ f^\infty_-(r)]$ to \eqref{radeq} for $r\in [0,\ \infty)$.  Moreover,
\begin{gather}
f^\infty_{\pm}\in C^{\infty}\left((0,\ \infty)\right),\\
f^\infty_{\pm}(r)>0\ \text{ for all}\ r>0,\label{fpstive}\\ 
f^\infty_{\pm}(r)\sim r\ \text{for}\ r\sim0, \label{fnear0} \\
\int_0^\infty [(f^\infty_\pm)']^2\, r\, dr <\infty, \label{gradint}
\end{gather} 
In particular, $\Psi(x)=[f^\infty_+(r)e^{i\tta},\ f^\infty_-(r)e^{i\tta}]$ is an entire solution of \eqref{Reqns} in $\mathbb{R}^2$.
\end{prop}

To determine the shape of the vortex profiles, we first consider the asymptotic form of the solutions for $r\to\infty$.  We prove that (see \cite{ag13}):
\begin{theorem}\label{asymptotics}
Let $[f^\infty_{+}, f^\infty_{-}]$ be the entire solution of \eqref{radeq}.  Then we have 
$$
f^\infty_{\pm}=\dis t_{\pm}+\frac{a_{\pm}}{r^2}+O(r^{-4}),\quad [f^\infty_{\pm}]'(r)=-\frac{2a_{\pm}}{r^3} + O(r^{-5}), \ \ \text{as}\ \ r\to\infty,
$$
with 
\begin{equation}\label{apm}
a_{\pm}=\dis \frac{1}{2}\frac{B-A_{\mp}}{(A_{+}A_{-}-B^{2})t_{\pm}},
\end{equation}
\end{theorem}

We also consider the question of monotonicity of the solution profiles. The validity of this property is strongly dependent on the value of $B$.  In  \cite{ag13} we prove:

\begin{theorem}\label{monotone}
Let $A_+, A_->0$ be fixed, and $B$ such that $B^2<A_+A_-$.
Assume $\Psi^\infty(x;B)=[f^\infty_{+}(r;B)e^{i\tta},f^\infty_{-}(r;B)e^{i\tta}]$ is the entire equivariant solution for those parameters $A_\pm,B$.
\begin{enumerate}[{\bf (i)}]
\item If $B<0$, then $[f^\infty_{\pm}]'(r;B)\ge 0$ for all $r>0$.


\item  There exists $B_0>0$ such that $[f^\infty_\pm]'(r;B)\ge 0$ for all $r>0$ and all $B$ with $0\le B\le B_0$.
\end{enumerate}
\end{theorem}

Note that by the asymptotic expansion derived in Theorem~\ref{asymptotics}, when $0<\min\{ A_+, A_-\}<B<\sqrt{A_+A_-}$,
one of the functions $f^\infty_\pm(r)$ approaches its limit value $t_\pm$ from above, and hence it cannot be monotone.  Indeed, the form of the coefficient $a_\pm$ in \eqref{apm} suggests that the optimum value of $B_0$ in Theorem~\ref{monotone} should be
$B_*:=\min\{ A_+, A_-\}$.  Unfortunately, the proof of part (ii) in Theorem~\ref{monotone} does not provide an explicit value for $B_0$, but we conjecture that monotonicity should indeed hold for $0<B<B_*$.

\section{Decomposition of $E''_\lambda (\Psi)$}

We begin with a decomposition of the second variation of energy $E_\lambda (\Psi)$ in \eqref{energylam}. The decomposition procedure follows that 
of \cite{m95} (for the classical GL equation) and \cite{abm13} (for the ``balanced" case), but we provide all details here for the general case.
%
The associated second variation of $E_\lambda (\Psi)$ around $\Psi=(\psi_{+}, \psi_{-})$ in direction $\Phi=(\phi_{+}, \phi_{-})\in H^{1}_{0}(\mathbb{D}_{1}; \mathbb{C}^{2})$ is 
\begin{multline}\label{2ndvarE}
E_{\lambda}''(\Psi)[\Phi]=\dis\int_{\mathbb{D}_{1}}|\nabla\Phi|^2 +\lambda [A_{+}(|\psi_{+}|^2-t^2_{+})+B(|\psi_{-}|^2-t^2_{-})] |\phi_{+}|^2\\
                              +\lambda [A_{-}(|\psi_{-}|^2-t^2_{-})+B(|\psi_{+}|^2-t^2_{+})] |\phi_{-}|^2\\
                              +2\lambda [A_{+}\langle\psi_{+}, \phi_{+}\rangle^{2}+A_{-}\langle\psi_{-}, \phi_{-}\rangle^{2}
                              +2B\langle\psi_{+}, \phi_{+}\rangle \langle\psi_{-}, \phi_{-}\rangle]\ ,
\end{multline}
with $\langle u,v\rangle=\text{Re}(\bar{u}v)={{\bar{u}v+u\bar{v}}\over 2}$.
Denote by $\Psi_{rad}=(f_+(r;\lambda)e^{i\theta}, f_-(r;\lambda)e^{i\theta})$, the unique equivariant solution to \eqref{dirichlet} in $\mathbb{D}_1$ with parameter $\lambda$.  We will often suppress the dependence on $\lambda$ to simplify notation, when its dependence on $\lambda$ is not an essential issue.

First, we note that for $\lambda$ small enough, there are no other solutions to Dirichlet problem \eqref{dirichlet}:
\begin{prop}
There exists $\lambda^*$ so that for every $\lambda<\lambda^*$ the unique solution to \eqref{dirichlet} is $\Psi=[\psi_+, \psi_- ]=[f_{+}(r)e^{i\theta}, f_{-}(r)e^{i\theta}]$.
\end{prop}
\begin{proof}
We first define the convex set $\mathcal{B}=\{\Psi \in \mathbf{H} : |\Psi|\le\Lambda\ \text{in}\ \mathbb{D}_1 \},$ with $\Lambda$ as defined in Lemma~\ref{uniformbound}. 
By Lemma~\ref{uniformbound} any solution of 
\eqref{dirichlet} lies in $\mathcal{B}$. For any $\Psi\in\mathcal{B}$ and together with the positive definite condition (H), we have
{\allowdisplaybreaks
\begin{align*}
E_{\lambda}''(\Psi)[\Phi]&\ge\int_{\mathbb{D}_1}|\nabla\Phi|^{2}+\lambda [A_{+}(|\psi_{+}|^2-t^2_{+})+B(|\psi_{-}|^2-t^2_{-})] |\phi_{+}|^2\\
                              &\qquad\qquad+\lambda [A_{-}(|\psi_{-}|^2-t^2_{-})+B(|\psi_{+}|^2-t^2_{+})] |\phi_{-}|^2\\
                              &\ge\int_{\mathbb{D}_1}|\nabla\Phi|^{2}-\lambda[(A_{+}t^2_{+}+|B|t^2_- )|\phi_{+}|^2 +(A_{-}t^2_{-}+|B|t^2_+ )|\phi_{-}|^2 ]\\
                              &\ge\int_{\mathbb{D}_1}|\nabla\Phi|^{2}-\lambda[A_{+}t^2_{+}+A_{-}t^2_{-}+|B|(t^2_+ +t^2_- )]|\Phi|^{2}\\
                              &=\int_{\mathbb{D}_1}|\nabla\Phi|^{2}-\lambda C|\Phi|^{2} ,
\end{align*}}
with constant $C=C(A_{\pm},|B|, t_\pm )\ge0$ independent of $\lambda, \Phi$. By choosing $\lambda^*$ sufficiently small, $C\lambda$ will be smaller than the first Dirichlet eigenvalue of the Laplacian in $\mathbb{D}_1$, and we may conclude that $E_{\lambda}''(\Psi)[\Phi]$ is a strictly positive definite quadratic form on 
$H^1_{0}(\mathbb{D}_1; \mathbb{C})$, for any $\Psi\in\mathcal{B}$. Thus, $E_\lambda (\Psi)$ is strictly convex on $\mathcal{B}$, and hence it has a unique critical point.
\end{proof}

The stability or instability of the symmetric vortex solutions will be determined by the sign of the ground-state eigenvalue of the linearization  of the energy around $\Psi$ in \eqref{2ndvarE}.
Motivated by \cite{m95}, we will decompose an arbitrary test function $\Phi$ in its Fourier modes, and write the linearized operator $E_{\lambda}''(\Psi)[\Phi] $ as a direct sum.

For $\Phi=[\phi_+, \phi_- ]$, each $\phi_{\pm}\in H^{1}_{0}(\mathbb{D}_{1}; \mathbb{C})$ can be written in its Fourier modes in $\theta$
\begin{equation*}
\phi_{\pm}=\dis\sum_{n\in\mathbb{Z}}b^{\pm}_{n}(r)e^{in\tta},
\end{equation*}
where $b^{\pm}_{n}(r)\in H^{1}_{\text{loc}}((0,1]; \mathbb{C})$. Using above Fourier decomposition, we may evaluate each term of the second variation
$E''_{\lambda}(\Psi)[\Phi]$.  First,
\begin{align*}
\int_{\mathbb{D}_{1}}|\nabla\Phi|^2 
  &=\int^{2\pi}_{0}\int^1_{0}\sum_{n\in\mathbb{Z}}\left\{|(b^{+}_{n})'|^{2}+|(b^{-}_{n})'|^{2}+|b^+_{n}|^{2}{n^{2}\over r^2 }+|b^-_{n}|^{2}{n^{2}\over r^2 }\right\}rdr\\
   &  =\int^{2\pi}_{0}\int^1_{0}\sum^{\infty}_{n=1}\left\{ |(b^+_{n+1})'|^{2}+|(b^-_{n+1})'|^{2}+{(n+1)^{2}\over r^2}|b^{+}_{n+1}|^{2}+{(n+1)^{2}\over r^2}|b^{-}_{n+1}|^{2}\right\}rdr\\
   &\qquad+\int^{2\pi}_{0}\int^1_{0}\sum^{\infty}_{n=0} \left\{ |(b^+_{1-n})'|^{2}+|(b^-_{1-n})'|^{2}+{(1-n)^{2}\over r^2}|b^{+}_{1-n}|^{2}+{(1-n)^{2}\over r^2}|b^{-}_{1-n}|^{2}\right\}rdr .
\end{align*}
Next,
{\allowdisplaybreaks
\begin{align*}
\int_{\mathbb{D}_1 }\langle\psi_{\pm}, \phi_{\pm}\rangle^{2}&=\dis \sum_{n\in\mathbb{Z}}\int^{2\pi}_{0}\int^1_{0}\langle f_{\pm}e^{i\tta}, b^{\pm}_{n}e^{in\tta}\rangle^2 rdrd\tta\\
   &={\pi\over 2}\int^1_{0}f^2_{\pm}\sum_{n\in\mathbb{Z}}|b^{\pm}_{n+1}+\overline{b^{\pm}_{1-n}}|^2 rdr\\
   &={\pi\over 2}\int^1_{0}f^2_{\pm}\left\{ \sum^{\infty}_{n=0}|b^{\pm}_{n+1}+\overline{b^{\pm}_{1-n}}|^2 +\sum^{-1}_{n=-\infty}|b^{\pm}_{n+1}+\overline{b^{\pm}_{1-n}}|^2 \right\}rdr\\
   &={\pi\over 2}\int^1_{0}f^2_{\pm}\left\{ \sum^{\infty}_{n=0}|b^{\pm}_{n+1}+\overline{b^{\pm}_{1-n}}|^2 +\sum^{\infty}_{n=1}|b^{\pm}_{n+1}+\overline{b^{\pm}_{1-n}}|^2\right\}rdr,
\end{align*}}
where by changing the index we obtain the last identity.

Using the orthogonality, we obtain
{\allowdisplaybreaks
\begin{align*}
&\int_{\mathbb{D}_{1}}\langle\psi_{+},\phi_{+}\rangle\langle\psi_{-},\phi_{-}\rangle\\
      &\qquad=\int^{2\pi}_{0}\int^1_{0}\sum_{n,m}\langle f_{+}e^{i\tta}, b^+_{n}e^{in\tta}\rangle\langle f_{-}e^{i\tta}, b^-_{m}e^{im\tta}\rangle rdrd\tta\\
      &\qquad=\sum_{n,m}{1\over 4}\int^{2\pi}_{0}\int^1_{0}f_{+}f_{-}[b^+_{n}b^-_{m}e^{i(m+n-2)\tta}+b^+_{n}\overline{b^-_{m}}e^{i(n-m)\tta}\\
      &\qquad\qquad\qquad +\overline{b^+_n}b^-_{m}e^{i(m-n)\tta}+\overline{b^+_{n}}\overline{b^-_{m}}e^{i(2-n-m)\tta}]rdrd\tta\\
      &\qquad=\sum_{n+m=2}{\pi\over 2}\int^1_{0}f_{+}f_{-}(b^+_n b^-_m +\overline{b^+_n}\overline{b^-_m})rdr
      +\sum_{n=m}{\pi\over 2}\int^1_{0}f_{+}f_{-}(b^+_n \overline{b^-_m} +\overline{b^+_n}b^-_{m})rdr\\
      &\qquad =\sum_{n\in\mathbb{Z}}{\pi\over 2}\int^1_{0}f_{+}f_{-}(b^+_n b^-_{2-n}+\overline{b^+_n}\overline{b^-_{2-n}}+b^+_n \overline{b^-_n}+\overline{b^+_n}b^-_n )rdr\\
      &\qquad =\sum_{n\in\mathbb{Z}}{\pi\over 2}\int^1_{0}f_{+}f_{-}(b^+_{1+n} b^-_{1-n}+\overline{b^+_{1-n}}\overline{b^-_{1+n}}+b^+_{1+n} \overline{b^-_{1+n}}+\overline{b^+_{1-n}}b^-_{1-n} )rdr\\
      &\qquad=\left(\sum^\infty_{n=0}+\sum^{-1}_{n=-\infty}\right)\left\{{\pi\over 2}\int^1_{0}f_{+}f_{-}(b^+_{1+n} b^-_{1-n}+\overline{b^+_{1-n}}\overline{b^-_{1+n}}+b^+_{1+n} \overline{b^-_{1+n}}+\overline{b^+_{1-n}}b^-_{1-n} )rdr\right\}\\
      &\qquad=\sum^\infty_{n=0}{\pi\over 2}\int^1_{0}f_{+}f_{-}(b^+_{1+n} b^-_{1-n}+\overline{b^+_{1-n}}\overline{b^-_{1+n}}+b^+_{1+n} \overline{b^-_{1+n}}+\overline{b^+_{1-n}}b^-_{1-n} )rdr\\
      &\qquad\qquad +\sum^\infty_{n=1}{\pi\over 2}\int^1_{0}f_{+}f_{-}(b^+_{1-n} b^-_{1+n}+\overline{b^+_{1+n}}\overline{b^-_{1-n}}+b^+_{1-n} \overline{b^-_{1-n}}+\overline{b^+_{1+n}}b^-_{1+n} )rdr ,
\end{align*}}
and,
{\allowdisplaybreaks
\begin{align*}
\int_{\mathbb{D}_1 }|\phi_{\pm}|^{2}&=\int^{2\pi}_0\int^1_0 \sum_{n\in\mathbb{Z}} |b^\pm_n e^{in\tta}|^2 rdrd\tta\\
                                                                &=\int^{2\pi}_0\int^1_0 \sum_{n\in\mathbb{Z}} |b^\pm_n |^2 rdrd\tta\\
                                                                &=2\pi \int^1_0 \sum_{n\in\mathbb{Z}}|b^\pm_n |^2 rdr\\
                                                                &=2\pi \int^1_0 \left\{\sum^\infty_{n=1}|b^\pm_{n+1}|^{2}+\sum^\infty_{n=0}|b^\pm_{1-n}|^{2} \right\}rdr.
\end{align*}}

Therefore, we have the following quadratic forms associated
to \eqref{2ndvarE}: 
{\allowdisplaybreaks
\begin{align*}
&\mathcal{Q}^{(n)}_{\lambda}(b^{\pm}_{1+n}, b^{\pm}_{1-n})\\
&\qquad\dis=2\pi\int^{1}_{0}\left[|(b^{+}_{1+n})'|^{2}+|(b^{+}_{1-n})'|^{2}+{(1+n)^{2}\over r^2}|b^{+}_{1+n}|^{2}+{(1-n)^{2}\over r^2}|b^{+}_{1-n}|^{2}\right]rdr\\
&\quad\qquad \dis +2\pi\int^{1}_{0}\left[|(b^{-}_{1+n})'|^{2}+|(b^{-}_{1-n})'|^{2}+{(1+n)^{2}\over r^2}|b^{-}_{1+n}|^{2}+{(1-n)^{2}\over r^2}|b^{-}_{1-n}|^{2}\right]rdr\\
               &\quad\qquad \dis +2\lambda\pi\int^1_{0}\left[A_{+}f^2_{+}|b^{+}_{1+n}+\overline{b^{+}_{1-n}}|^{2}+A_{-}f^2_{-}|b^{-}_{1+n}+\overline{b^{-}_{1-n}}|^{2}\right]rdr\\
               &\quad\qquad  \dis +4\lambda\pi B\int^1_{0}f_{+}f_{-}[\langle b^{+}_{1+n}+\overline{b^{+}_{1-n}}, b^-_{1+n}\rangle+\langle b^+_{1-n}+\overline{b^{+}_{1+n}}, b^{-}_{1-n}\rangle]rdr\\
              &\quad\qquad \dis +2\lambda\pi\int^1_{0}[A_{+}(f^2_{+}-t^2_{+})+B(f^2_{-}-t^2_{-})](|b^{+}_{1+n}|^{2}+|b^+_{1-n}|^{2})rdr\\
               &\quad\qquad \dis +2\lambda\pi\int^1_{0}[A_{-}(f^2_{-}-t^2_{-})+B(f^2_{+}-t^2_{+})](|b^{-}_{1+n}|^{2}+|b^-_{1-n}|^{2})rdr
\end{align*}}
for $n\neq0$, and
{\allowdisplaybreaks
\begin{align*}
&\mathcal{Q}^{(0)}_{\lambda}(b^{\pm}_{1})\\
&\quad=\dis 2\pi\int^1_{0}\left\{|(b^{+}_{1})'|^{2}+|(b^{-}_{1})'|^{2}+{1\over r^2}(|b^+_{1}|^{2}+|b^-_{1}|^{2})\right\}rdr\\
&\qquad \dis +\lambda\pi\int^1_{0}\left\{A_{+}f^2_{+}|b^{+}_{1}+\overline{b^{+}_{1}}|^{2}+A_{-}f^2_{-}|b^{-}_{1}+\overline{b^{-}_{1}}|^{2}+4Bf_{+}f_{-}\langle b^+_{1}+\overline{b^+_1}, b^-_{1}\rangle\right\}rdr\\
&\qquad\dis +2\lambda\pi\int^1_{0}\left\{[A_{+}(f^2_{+}-t^2_{+})+B(f^2_{-}-t^2_{-})]|b^{+}_{1}|^{2}+[A_{-}(f^2_{-}-t^2_{-})+B(f^2_{+}-t^2_{+})]|b^{-}_{1}|^{2}\right\}rdr\\
&\quad=\dis 2\pi\int^1_{0}\left\{|(b^{+}_{1})'|^{2}+|(b^{-}_{1})'|^{2}+{1\over r^2}(|b^+_{1}|^{2}+|b^-_{1}|^{2})\right\}rdr\\
&\qquad+2\lambda\pi\int^1_{0}\left\{[A_{+}(f^2_{+}-t^2_{+})+B(f^2_{-}-t^2_{-})]|b^{+}_{1}|^{2}+[A_{-}(f^2_{-}-t^2_{-})+B(f^2_{+}-t^2_{+})]|b^{-}_{1}|^{2}\right\}rdr\\
&\qquad+\dis 4\lambda\pi\int^1_{0}\left\{A_{+}f^2_{+}(\text{Re}b^+_{1})^{2}+A_{-}f^2_{-}(\text{Re}b^-_{1})^{2}+2Bf_{+}f_{-}\text{Re}(b^+_{1})\text{Re}(b^-_{1})\right\}rdr .
\end{align*} }
Therefore, we can represent \eqref{2ndvarE} as a sum of decoupled quadratic forms,
\begin{equation}\label{2ndvarEreform}
E''_{\lambda}(\Psi)[\Phi]=\dis \mathcal{Q}^{(0)}_{\lambda}(b^{\pm}_{1})+\mathcal{Q}^{(1)}_{\lambda}(b^{\pm}_{2}, b^{\pm}_{0})+\sum^{\infty}_{n=2}\mathcal{Q}^{(n)}_{\lambda}(b^{\pm}_{1+n}, b^{\pm}_{1-n}) .
\end{equation}

Consquently, the operator $\mathcal{L}_{\lambda}$ associated to $E''_{\lambda}(\Psi)[\Phi]$ can be identified with a direct sum of self-adjoint operators, acting on blocks of Fourier modes,
\begin{equation*}
\mathcal{L}_{\lambda}(\Phi)\cong\bigoplus^\infty_{n=0}\mathcal{L}^{(n)}_{\lambda}(b^\pm_{n+1}, b^\pm_{1-n}),
\end{equation*}
where the operators $\mathcal{L}^{(n)}_\lambda$ are associated to the quadratic forms $\mathcal{Q}^{(n)}_\lambda$.  Define (in Fourier Space)
$$
\widetilde{\mathcal{L}}_{\lambda}\widetilde{\Phi}:=\bigoplus_{n\neq1}\mathcal{L}^{(n)}_{\lambda}(b^\pm_{n+1}, b^{\pm}_{1-n}),
$$
where $\widetilde{\Phi}=[\widetilde{\phi}_{+}, \widetilde{\phi}_{-}]=\left[\sum_{n\neq0,2}b^{+}_{n}e^{in\theta}, \sum_{n\neq0,2}b^{-}_{n}e^{in\theta}\right]$, and so 
$\mathcal{L}_{\lambda}(\Phi)\cong\mathcal{L}^{(1)}_{\lambda}(b^{\pm}_{2}, b^\pm_{0})\oplus\widetilde{\mathcal{L}}_{\lambda}\widetilde{\Phi} $, with $\tilde{\mathcal{Q}}_{\lambda}$ denote the quadratic form associated to $\tilde{\mathcal{L}}_{\lambda}$. 

From the above computations and definitions we conclude the following result:

\begin{prop}
We have $\mathcal{L}_{\lambda}\Phi\cong\mathcal{L}^{(1)}_{\lambda}(b^\pm_2, b^\pm_0 )\oplus\widetilde{\mathcal{L}}_{\lambda}\widetilde{\Phi}$, where the operators $\mathcal{L}^{(n)}_\lambda$ are associated to the quadratic forms $\mathcal{Q}^{(n)}_\lambda$, the operator $\widetilde{\mathcal{L}}_{\lambda}$  is associated to the quadratic form $\widetilde{\mathcal{Q}}_{\lambda}$ and $\widetilde{\Phi}=[\widetilde{\phi}_{+}, \widetilde{\phi}_{-}]=\left[\sum_{n\neq0,2}b^{+}_{n}e^{in\theta}, \sum_{n\neq0,2}b^{-}_{n}e^{in\theta}\right]$.
\end{prop} 

We now show that the ground state eigenvalue of $\mathcal{L}_{\lambda}$ is determined by only one of the operators, $\mathcal{L}_{\lambda}^{(1)}$.  We proceed by a series of reductions.
Define 
$$a^{\pm}_{1}:=\dis i\left(\sum_{ {n\neq0,2}\atop{n\ge0} }|b^{\pm}_{n}|^2\right)^{1/2}.
$$
Then, $a_{1}(r)$ is purely imaginary and we have 
\begin{equation*}
|a^{\pm}_{1}|^{2}=\dis \left|i\left(\sum_{ {n\neq0,2}\atop{n\ge0} }|b^{\pm}_{n}|^{2}\right)^{1\over2}\right|^{2}=\sum_{ {n\neq0,2}\atop{n\ge0} }|b^{\pm}_{n}|^{2} .
\end{equation*}
By the facts
\begin{equation*}
(a^{\pm}_{1})'=\dis i\left(\sum_{ {n\neq0,2}\atop{n\ge0} }|b^{\pm}_{n}|^2\right)^{-{1\over 2}}\sum_{ {n\neq0,2}\atop{n\ge0} }\langle b^{\pm}_{n}, (b^{\pm}_{n})' \rangle ,
\end{equation*}
and
\begin{equation*}
\dis |(a^{\pm}_{1})'|^{2}=\dis \frac{|\sum_{ {n\neq0,2}\atop{n\ge0} }\langle b^{\pm}_{n}, (b^{\pm}_{n})' \rangle|^{2}}{\sum_{ {n\neq0,2}\atop{n\ge0} }|b^{\pm}_{n}|^2}\le 
\frac{\sum_{ {n\neq0,2}\atop{n\ge0} }|b^{\pm}_{n}|^{2}\sum_{ {n\neq0,2}\atop{n\ge0} }|(b^{\pm}_{n})'|^2}{\sum_{ {n\neq0,2}\atop{n\ge0} }|b^{\pm}_{n}|^2}
=\sum_{ {n\neq0,2}\atop{n\ge0} }|(b^{\pm}_{n})'|^2 ,
\end{equation*}
we obtain that
\begin{equation}\label{a1ineq}
|(a^{\pm}_{1})'|^{2}\le\dis \sum_{ {n\neq0,2}\atop{n\ge0} }|(b^{\pm}_{n})'|^{2} .
\end{equation}

On the other hand, by the positive definite condition (H), we conclude that
\begin{multline*}
A_{+}f^2_{+}|b^{+}_{1+n}+\overline{b^{+}_{1-n}}|^{2}+A_{-}f^2_{-}|b^{-}_{1+n}+\overline{b^{-}_{1-n}}|^{2}\\
+2Bf_{+}f_{-}[\langle b^{+}_{1+n}+\overline{b^{+}_{1-n}},b^{-}_{1+n}\rangle+\langle b^{+}_{1-n}+\overline{b^{+}_{1+n}},b^{-}_{1-n}\rangle]\\
=A_{+}f^2_{+}|b^{+}_{1+n}+\overline{b^{+}_{1-n}}|^{2}+A_{-}f^2_{-}|b^{-}_{1+n}+\overline{b^{-}_{1-n}}|^{2}\\
+2Bf_{+}f_{-}\text{Re}[(b^{+}_{1+n}+\overline{b^{+}_{1-n}})(\overline{b^{-}_{1+n}}+b^{-}_{1-n})]
      \ge 0,
\end{multline*}
for each $n>0$ and $n\neq1$.
Similarly, we also have 
$$
A_{+}f^2_{+}(\text{Re}b^+_{1})^{2}+A_{-}f^2_{-}(\text{Re}b^-_{1})^{2}+2Bf_{+}f_{-}(\text{Re}b^+_{1})(\text{Re}b^-_{1})
\ge 0.
$$

Therefore, for $n\neq 1$, we have that
{\allowdisplaybreaks
\begin{align}\label{Qnineq}
&\mathcal{Q}^{(n)}_{\lambda}(b^{\pm}_{1+n}, b^{\pm}_{1-n})\nnn\\
                 &\qquad\ge \dis 2\pi\int^1_{0}\left\{|(b^{+}_{1+n})'|^{2}+|(b^{+}_{1-n})'|^{2}+{(1+n)^{2}\over r^{2}}|b^{+}_{1+n}|^{2}+{(1-n)^{2}\over r^{2}}|b^{+}_{1-n}|^{2}\right\}rdr\nnn\\
                  &\quad\qquad \dis + 2\pi\int^1_{0}\left\{|(b^{-}_{1+n})'|^{2}+|(b^{-}_{1-n})'|^{2}+{(1+n)^{2}\over r^{2}}|b^{-}_{1+n}|^{2}+{(1-n)^{2}\over r^{2}}|b^{-}_{1-n}|^{2}\right\}rdr\nnn\\
                   &\quad\qquad  \dis +2\lambda\pi\int^1_{0}[A_{+}(f^2_{+}-t^2_{+})+B(f^2_{-}-t^2_{-})](|b^{+}_{1+n}|^{2}+|b^{+}_{1-n}|^{2})rdr\nnn\\
                   &\quad\qquad \dis +2\lambda\pi\int^1_{0}[A_{-}(f^2_{-}-t^2_{-})+B(f^2_{+}-t^2_{+})](|b^{-}_{1+n}|^{2}+|b^{-}_{1-n}|^{2})rdr .
\end{align}}
By \eqref{a1ineq} and \eqref{Qnineq}, it follows that
\begin{align}\label{Qa1ineq}
\dis \widetilde{\mathcal{Q}}_{\lambda}(\tilde{\Phi})&=\sum_{{n\neq 1}\atop {n\ge0}}\mathcal{Q}^{(n)}_{\lambda}(b^{\pm}_{1+n}, b^{\pm}_{1-n})\nnn\\
                                 &\ge \dis 2\pi\int^1_{0} \left\{|(a^+_1)'|^{2}+|(a^-_1)'|^{2}+{1\over r^2}(|a^+_1|^{2}+|a^-_1|^{2})\right\}rdr\nnn\\
                                 &\quad \dis +2\lambda\pi \int^1_{0}[A_{+}(f^2_{+}-t^2_{+})+B(f^2_{-}-t^2_{-})]|a^+_1|^{2}rdr\nnn\\
                                 &\quad \dis +2\lambda\pi \int^1_{0}[A_{-}(f^2_{-}-t^2_{-})+B(f^2_{+}-t^2_{+})]|a^-_1|^{2}rdr\nnn\\
                                 &=:Q^{(0)}_{\lambda}(a^{\pm}_{1}) .
\end{align}

Meanwhile, we have that
{\allowdisplaybreaks
\begin{align*}
&\mathcal{Q}^{(1)}_{\lambda}(b^{\pm}_{2}, b^{\pm}_{0})\\
&\qquad\dis=2\pi\int^{1}_{0}\left[|(b^{+}_{2})'|^{2}+|(b^{+}_{0})'|^{2}+{4\over r^2}|b^{+}_{2}|^{2}\right]rdr\\
&\quad\qquad \dis +2\pi\int^{1}_{0}\left[|(b^{-}_{2})'|^{2}+|(b^{-}_{0})'|^{2}+{4\over r^2}|b^{-}_{2}|^{2}\right]rdr\\
               &\quad\qquad \dis +2\lambda\pi\int^1_{0}\left[A_{+}f^2_{+}|b^{+}_{2}+\overline{b^{+}_{0}}|^{2}+A_{-}f^2_{-}|b^{-}_{2}+\overline{b^{-}_{0}}|^{2}\right]rdr\\
               &\quad\qquad  \dis +4\lambda\pi B\int^1_{0}f_{+}f_{-}[\langle b^{+}_{2}+\overline{b^{+}_{0}}, b^-_{2}\rangle+\langle b^+_{0}+\overline{b^{+}_{2}}, b^{-}_{0}\rangle]rdr\\
              &\quad\qquad \dis +2\lambda\pi\int^1_{0}[A_{+}(f^2_{+}-t^2_{+})+B(f^2_{-}-t^2_{-})](|b^{+}_{2}|^{2}+|b^+_{0}|^{2})rdr\\
               &\quad\qquad \dis +2\lambda\pi\int^1_{0}[A_{-}(f^2_{-}-t^2_{-})+B(f^2_{+}-t^2_{+})](|b^{-}_{2}|^{2}+|b^-_{0}|^{2})rdr.
\end{align*}}
The self-adjoint operator associated to $\mathcal{Q}^{(1)}_{\lambda}(b^{\pm}_{2}, b^{\pm}_{0})$ is 
\begin{equation}\label{L1complex}
\mathcal{L}^{(1)}_{\lambda}\left[\begin{array}{l}b^\pm_0 \\b^\pm_2 \end{array}\right]
=\left[\begin{array}{c}-(b^\pm_0 )''-{1\over r}(b^\pm_0 )'+\lambda[A_{\pm}(f^2_{\pm}-t^2_{\pm})+B(f^2_{\mp}-t^2_{\mp})]b^\pm_{0}\qquad\quad\\
\qquad\qquad\qquad+\lambda A_{\pm}f^2_{\pm}(\overline{b^\pm_2}+b^\pm_0 )+\lambda Bf_+ f_-(\overline{b^\mp_2}+b^\mp_0 )\\
-(b^\pm_2 )''-{1\over r}(b^\pm_2 )'+{4\over r^2}b^\pm_{2}+\lambda[A_{\pm}(f^2_{\pm}-t^2_{\pm})+B(f^2_{\mp}-t^2_{\mp})]b^\pm_{2}\\
\qquad\qquad\qquad+\lambda A_{\pm}f^2_{\pm}(b^\pm_2 +\overline{b^\pm_0 })+\lambda Bf_+ f_-(b^\mp_2 +\overline{b^\mp_0 })
\end{array}\right] .
\end{equation}
We perform a further reduction of the operator $\mathcal{L}^{(1)}_{\lambda}$, which will be useful in proving instability for $B>0$: define a quadratic form $Q^{(1)}_\lambda$ on real-valued radial functions $(a^\pm_2, a^\pm_0 )$ by 
\begin{align}\label{Q(1)}
&Q^{(1)}_{\lambda}(a^{\pm}_{2},a^{\pm}_{0})=\dis 2\pi\int^1_{0}\sum_{i=\pm}\left[|(a^{i}_{2})'|^{2}+|(a^{i}_{0})'|^{2}+{4\over r^{2}}|a^{i}_{2}|^{2}\right]rdr\\
\nnn
                                                      &\qquad+\dis 2\lambda\pi\int^1_{0}[A_{+}(f^2_{+}-t^2_{+})+B(f^2_{-}-t^2_{-})](|a^+_{2}|^{2}+|a^+_{0}|^{2})rdr\\
                                                      \nnn
                                                       &\qquad+\dis 2\lambda\pi\int^1_{0}[A_{-}(f^2_{-}-t^2_{-})+B(f^2_{+}-t^2_{+})](|a^-_{2}|^{2}+|a^-_{0}|^{2})rdr\\   \nnn
                                                      &\qquad+\dis  2\lambda\pi\int^1_{0}[A_{+}f^2_{+}(a^+_{0}-a^+_{2})^{2}+A_{-}f^2_{-}(a^-_{0}-a^-_{2})^{2}
                                                      +2Bf_{+}f_{-}(a^+_{0}-a^+_{2})(a^-_{0}-a^-_{2})]rdr .
\end{align}
The associated self-adjoint operator to $Q^{(1)}_\lambda$ is 
\begin{equation}\label{L1real}
\mathcal{M}^{(1)}_{\lambda}
\left[\begin{array}{c}a^\pm_0 \\a^\pm_2 \end{array}\right]=
\left[\begin{array}{l}-(a^\pm_{0})''-{1\over r}(a^\pm_{0})'+\lambda [A_{\pm}(f^2_{\pm}-t^2_{\pm})+B(f^2_{\mp}-t^2_{\mp})]a^{\pm}_{0}\\
\qquad\qquad\qquad+\lambda A_{\pm}f^2_{\pm}(a^\pm_{0}-a^\pm_{2})+\lambda Bf_{+}f_{-}(a^\mp_{0}-a^\mp_{2})\\
 -(a^\pm_{2})''-{1\over r}(a^\pm_{2})'+{4\over r^2 }a^\pm_{2}+\lambda [A_{\pm}(f^2_{\pm}-t^2_{\pm})+B(f^2_{\mp}-t^2_{\mp})]a^{\pm}_{2}\\
\qquad\qquad\qquad+ \lambda A_{\pm}f^2_{\pm}(a^\pm_{2}-a^\pm_{0})+\lambda Bf_{+}f_{-}(a^\mp_{2}-a^\mp_{0})
 \end{array}\right].
\end{equation}

Denote 
$$
\tilde{\mu}_\lambda =\dis\min_{\|\Phi\|^2_{L^{2}}=1}E_{\lambda}''(\Psi)[\Phi] ,\qquad \mu_{\lambda}=\dis\min_{\|(a^\pm_{0}, a^\pm_{1}, a^\pm_{2})\|^2_{L^{2}}=1}Q_{\lambda}(a^\pm_{0}, a^\pm_{1}, a^\pm_{2}) ,
$$ 
with $Q_{\lambda}(a^\pm_{0}, a^\pm_{1}, a^\pm_{2})=Q^{(0)}_{\lambda}(a^\pm_1 )+Q^{(1)}_{\lambda}(a^{\pm}_{2},a^{\pm}_{0})$, and associated self-adjoint operators $\mathcal{L}_{\lambda}$, $\mathcal{M}^{(n)}_\lambda (n=0, 1)$, respectively. We then write
$$
\|(a^\pm_{0}, a^\pm_{1}, a^\pm_{2})\|^2_{L^{2}}:=\dis\sum_{\sigma=\pm}\left[\|a^\sigma_0\|^2_{L^2 }+\|a^\sigma_1\|^2_{L^2 }+\|a^\sigma_2\|^2_{L^2 }\right].
$$
The following Proposition summarizes the outcome of the above reductions.  In particular, the stability or instability of the equivariant solution will depend on only two of the above operators, involving only three Fourier modes (in all, six radial functions $a_k^\pm$, $k=0,1,2$):

\begin{prop} \label{muequal}
We have $\tilde{\mu}_\lambda = \mu_{\lambda}, \forall\ \lambda>0$.
\end{prop}

\begin{proof}
Assume that $\Phi$ attains the minimum $\tilde{\mu}_\lambda$ under the constraint $\|\Phi\|^2_{L^2 }=1$. Then, by the Fourier decomposition of $\Phi$ and \eqref{2ndvarEreform}, 
\begin{align*}
\tilde{\mu}_\lambda=E_{\lambda}''(\Psi)[\Phi] 
&=\dis \mathcal{Q}^{(0)}_{\lambda}(b^{\pm}_{1})+\mathcal{Q}^{(1)}_{\lambda}(b^{\pm}_{2}, b^{\pm}_{0})+\sum^{\infty}_{n=2}\mathcal{Q}^{(n)}_{\lambda}(b^{\pm}_{1+n}, b^{\pm}_{1-n})\\
&\ge Q^{(0)}_{\lambda}(a^\pm_1 )+Q^{(1)}_{\lambda}(a^\pm_2, a^\pm_0 ) ,
\end{align*}
where $\|(a^\pm_{0}, a^\pm_{1}, a^\pm_{2})\|^2_{L^{2}}=1$ by the choice of $a^\pm_{0}, a^\pm_{1}, a^\pm_{2}$ as above. Therefore, $\tilde{\mu}_\lambda \ge\mu_\lambda$.

Conversely, if $(a^\pm_{0}, a^\pm_{1}, a^\pm_{2})$ attain the minimum for $\mu_\lambda$, we have $\Phi=[\phi_+, \phi_- ]$ with $\phi_{\pm}=a^\pm_0 +ia^\pm_1 e^{i\theta}-a^\pm_2 e^{2i\theta}$ and $\|\Phi\|^2_{L^2 }=\|(a^\pm_{0}, a^\pm_{1}, a^\pm_{2})\|^2_{L^{2}}=1$. Hence, 
$$
\tilde{\mu}_\lambda \le E_{\lambda}''(\Psi)[\Phi] =Q^{(0)}_{\lambda}(a^\pm_1 )+Q^{(1)}_{\lambda}(a^\pm_2, a^\pm_0 )=\mu_\lambda ,
$$
i.e. $\tilde{\mu}_\lambda\le \mu_\lambda$.  Hence, $\tilde{\mu}_\lambda = \mu_{\lambda}$.
\end{proof}

\section{Stability}

Our main goal in this section is to prove part (i) of Theorem~\ref{main}, which concerns the case $0>B>-\sqrt{A_+A_-}$.
From Proposition~\ref{muequal} and the Fourier decomposition, it suffices to show the positivity of the two ground-state eigenvalues,
$$
\mu^{(0)}_{\lambda}=\min_{\|a^\pm_{1}\|^2_{L^{2}}=1}Q^{(0)}_{\lambda}(a^\pm_1 ), \qquad
\mu^{(1)}_{\lambda}=\min_{\|(a^\pm_{0}, a^\pm_{2})\|^2_{L^{2}}=1}Q^{(1)}_{\lambda}(a^\pm_2, a^\pm_0 ) .
$$
Indeed, Proposition~\ref{muequal} implies that $\mu_\lambda=\min\{\mu^{(0)}_{\lambda}, \mu^{(1)}_{\lambda} \}$, so we must show that both $\mu^{(0)}_{\lambda}$ and $\mu^{(1)}_{\lambda}$ are positive.

In fact, for any $B$ (regardless of sign), a negative first eigenvalue can only be due to the quadratic form $Q_\lambda^{(1)}$, and never the form
$Q_\lambda^{(0)}$:

\begin{prop}\label{mu0}
$\mu^{(0)}_{\lambda}>0$ for any $B$ with $|B|<\sqrt{A_+A_-}$. 
\end{prop}

First, we require the following adaptation of a result by Mironescu \cite{m95}.

\begin{lemma}\label{4}
$\Psi_{\text{rad}}$ is the only minimizer of $E_{\lambda}$ in the class
\begin{equation*}
\mathcal{E}=\{ V=(g_{+}(r)e^{i\tta}, g_{-}(r)e^{i\tta})\big| V\in H^{1}(B_{1}; \mathbb{C}), g_{\pm}(1)=f_{\pm}(1)=t_{\pm}\}.
\end{equation*}
\end{lemma}

\begin{proof}
We have that
\begin{align*}
E_{\lambda}(V)&=\dis \pi\int^1_{0}[|g'_{+}|^{2}+|g'_{-}|^{2}+{1\over r^{2}}(|g_{+}|^{2}+|g_{-}|^{2})]rdr\\
                &\qquad+\dis {\lambda\pi\over2}\int^1_{0}[A_{+}(|g_{+}|^{2}-t^2_{+})^{2}+A_{-}(|g_{-}|^{2}-t^2_{-})^{2}+2B(|g_{+}|^{2}-t^2_{+})(|g_{-}|^{2}-t^2_{-})]rdr.
\end{align*}

If $\widetilde{V}=(|g_{+}|e^{i\tta}, |g_{-}|e^{i\tta})$, by the fact that $|\nabla|g_{\pm}||\le |\nabla g_{\pm}|$ for $\forall g_{\pm}\in\mathbb{C}$,
we have that $E_{\lambda}(\widetilde{V})\le E_{\lambda}(V)$. Hence, 
if $V$ is a minimizer, so is $\widetilde{V}$. Then $\widetilde{V}$ is smooth, which implies $g_{\pm}(r)\neq 0$ for $r\in (0, 1)$, and the equality occurs if $g_{\pm}\in\mathbb{R}$. 
From above analysis, the minimum of $E_\lambda$ in $\mathcal{E}$ is attained by a function $g_{\pm}(r)e^{i\tta}$ with $g_{\pm}(r)\ge0$. But from the uniqueness 
result in Proposition~\ref{existR}, $f_{\pm}$
are the only nonnegative solutions of 
\begin{equation*}
\left\{
 \begin{array}{l}
  \dis -g''_{\pm}-{g'_{\pm}\over r}+{g_{\pm}\over r^{2}}=\lambda[A_{\pm}(t^2_{\pm}-f^2_{\pm})+B(t^2_{\mp}-f^2_{\mp})]g_{\pm},\\
  \dis g_{\pm}(1)=f_{\pm}(1).
 \end{array}
\right.
\end{equation*}
Then by the uniqueness of above ODEs, we have $g_{\pm}\equiv f_{\pm}$. Therefore, 
$\Psi_{\text{rad}}
=(f_{+}(r;\lambda)e^{i\tta}, f_{-}(r;\lambda)e^{i\tta})$ is the {\it only}
minimizer of $E_\lambda$ in the class $\mathcal{E}$.
\end{proof}

\begin{proof}[Proof of Proposition~\ref{mu0}]
First, by Lemma~\ref{4}, we have
$$E''_{\lambda}(\Psi_{\text{rad}})[w_{+}, w_{-}]\ge 0$$
for $w=(w_{+}, w_{-})\in \mathcal{F}=\{v=(g_{+}(r)e^{i\tta}, g_{-}(r)e^{i\tta})\big| v\in H^1_{0}(B_{1}; \mathbb{C}), g_{\pm}(1)=0\}$.
We thus have
\begin{equation*}
\qquad\qquad \mu^{(0)}_\lambda = \min_{a^{\pm}_{1}\in H^{1}_{\text{loc}}((0,1];[0,\infty)), \atop \int^1_{0}(|a^+_{1}|^{2}+|a^-_{1}|^{2})rdr=1}
E''_{\lambda}(\Psi_{\text{rad}})(ia^{\pm}_{1})
\ge \dis \min_{w=(w_{+}, w_{-})\in\mathcal{F}}E''_{\lambda}(\Psi_{\text{rad}})(w_{\pm})\ge0.\qquad\qquad
\end{equation*}

We claim that $\mu^{(0)}_\lambda>0$. Suppose not: then there exist $w_{\pm}=ia^\pm_{1}(r)e^{i\tta}$, with $a^{\pm}_{1}\ge0$, $\dis \int^1_{0}(|a^+_{1}|^{2}+|a^-_{1}|^{2})rdr=1$ and $E''_{\lambda}(\Psi_{\text{rad}})(w_{\pm})=0$. It follows that $w=(w_{+}, w_{-})$ is a global minimizer of $E''_{\lambda}$, and hence verifies the equations:
$$
-\Delta w_{\pm}+\lambda[A_{\pm}(f^2_{\pm}-t^2_{\pm})+B(f^2_{\mp}-t^2_{\mp})]w_{\pm}+2\lambda[A_{\pm}\langle\psi_{\pm},w_{\pm}\rangle+B\langle\psi_{\mp},w_{\mp}\rangle]\psi_{\pm}=0.
$$

By the fact that $\langle f_{\pm}, ia^\pm_{1}\rangle=0$ with $a^\pm_{1}\in H^{1}(B_{1}; \mathbb{R})$, we have $E''_{\lambda}(\Psi_{\text{rad}})(ia^{\pm}_{1})=Q^{(0)}_{\lambda}(a^\pm_1 )$. The Euler-Lagrange equations associated to $Q^{(0)}_{\lambda}(a^\pm_1 )$ are
\begin{equation}\label{a1eqns}
\left\{
\begin{array}{l}
\dis -(a^{\pm}_{1})''-{(a^{\pm}_{1})'\over r}+{a^{\pm}_{1}\over r^{2}}=-\lambda[A_{\pm}(f^2_{\pm}-t^2_{\pm})+B(f^2_{\mp}-t^2_{\mp})]a^{\pm}_{1},\ \ \text{in}\ [0, 1],\\
a^{\pm}_{1}(0)=a^{\pm}_{1}(1)=0.
\end{array}
\right.
\end{equation}

Multiplying the $a^{\pm}_{1}$-equations of \eqref{a1eqns} by $rf_{\pm}$ respectively, and integrating by parts, we obtain that
\begin{align}\label{a1fpm}
\dis \int^1_{0}\left[-(a^{\pm}_{1})''-{(a^{\pm}_{1})'\over r}+{a^{\pm}_{1}\over r^{2}}\right]f_{\pm}rdr&=\dis -(a^{\pm}_{1})'(1)t_{\pm}+\int^1_{0}\left[(a^{\pm}_{1})'f'_{\pm}+{a^{\pm}_{1}f_{\pm}\over r^{2}}\right]rdr\nnn\\
&=\dis \lambda\int^1_{0} [A_{\pm}(f^2_{\pm}-t^2_{\pm})+B(f^2_{\mp}-t^2_{\mp})]a^{\pm}_{1}f_{\pm}rdr.
\end{align}

Also, we have that $f_{\pm}(r)$ satisfy the following equations
\begin{equation}\label{fpmeqns}
\left\{
\begin{array}{l}
\dis -f''_{\pm}-{f'_{\pm}\over r}+{f_{\pm}\over r^{2}}=-\lambda[A_{\pm}(f^2_{\pm}-t^2_{\pm})+B(f^2_{\mp}-t^2_{\mp})]f_{\pm},\ \ \text{in}\ [0, 1],\\
f_{\pm}(1)=t_{\pm}.
\end{array}
\right.
\end{equation}

After multiply $ra^{\pm}_{1}$ to $f_{\pm}$-equation of \eqref{fpmeqns} respectively, and integrate by parts, we similarly get that
\begin{equation}\label{fa1pm}
\dis -f'_{\pm}(1)a^{\pm}_{1}(1)+\int^1_{0}\left[(a^{\pm}_{1})'f'_{\pm}+{a^{\pm}_{1}f_{\pm}\over r^{2}}\right]rdr=\dis \lambda\int^1_{0} [A_{\pm}(f^2_{\pm}-t^2_{\pm})+B(f^2_{\mp}-t^2_{\mp})]a^{\pm}_{1}f_{\pm}rdr.
\end{equation}
Therefore, by \eqref{a1fpm} and \eqref{fa1pm}, we obtain that
$$
-f'_{\pm}(1)a^{\pm}_{1}(1)=-(a^{\pm}_{1})'(1)t_{\pm},
$$
which implies that $(a^{\pm}_{1})'(1)=0$. Together with $a^{\pm}_{1}(1)=0$, by the uniqueness of ODEs, it yields that $a^{\pm}_{1}(r)\equiv0$, which is a contradiction. We conclude that $\mu^{(0)}_\lambda >0$, which completes the proof of Proposition~\ref{mu0}.
\end{proof}

\begin{lemma}\label{mu1Bneg}
If $B<0$, then $\mu^{(1)}_\lambda >0$.
\end{lemma}
\begin{proof}  For any admissible $(a^\pm_0, a^\pm_2)$, define  
\begin{equation}\label{FK}
F_{\pm}:={1\over2}(a^{\pm}_{0}+a^{\pm}_{2}), \qquad
  K_{\pm}:={1\over2}(a^{\pm}_{0}-a^{\pm}_{2}).
\end{equation}
 We then rewrite $Q^{(1)}_\lambda (a^\pm_2 , a^\pm_0 )$ in terms of $F_{\pm}$ and $K_{\pm}$:
\begin{align*}
&\widehat{Q}^{(1)}_\lambda (F_\pm , K_\pm )=Q^{(1)}_\lambda (a^\pm_2 , a^\pm_0 )\\
&\qquad\qquad =4\pi\int^1_{0} \left[ |F'_\pm |^{2}+|K'_\pm |^{2}+{2\over r^2 }|F_{\pm}-K_{\pm}|^{2}\right]rdr\\
&\qquad\qquad\quad+8\lambda\pi\int^1_{0}[ A_{+}f^2_{+}K^2_{+}+A_{-}f^2_{-}K^2_{-}+2Bf_{+}f_{-}K_{+}K_{-}]rdr\\
&\qquad\qquad\quad +4\lambda\pi\int^1_{0}[A_{+}(f^2_{+}-t^2_{+})+B(f^2_{-}-t^2_{-})](F^2_{+}+K^2_{+})rdr\\
&\qquad\qquad\quad+4\lambda\pi\int^1_{0}[A_{-}(f^2_{-}-t^2_{-})+B(f^2_{+}-t^2_{+})](F^2_{-}+K^2_{-})rdr .
\end{align*}

The quantity $\widehat{Q}^{(1)}_\lambda (F_\pm , K_\pm )$ decreases if we replace $F_\pm , K_\pm $ by $|F_\pm | , |K_\pm | $, i.e. we have 
$$
\widehat{Q}^{(1)}_\lambda (|F_\pm |, |K_\pm | )\le \widehat{Q}^{(1)}_\lambda (F_\pm , K_\pm )=Q^{(1)}_\lambda (a^\pm_2, a^\pm_0 ). 
$$
Since this operation does not change the $L^2$ norm, 
$\|(F_\pm, K_\pm)\|_{L^2}=\|(|F_\pm|, |K_\pm|)\|_{L^2}$, we may conclude that the eigenvectors corresponding to the lowest eigenvalue of $\widehat{Q}^{(1)}_\lambda$, $(F_\pm , K_\pm )$ must satisfy $F_\pm \ge0$ and $K_\pm \ge0$. 

In the following, we take $(a^\pm_2 , a^\pm_0 )$ to be the ($L^2$-normalized) eigenfunctions corresponding to the ground state eigenvalue $\mu^{(1)}_\lambda$, and $K_{\pm},F_{\pm}$ associated to these eigenfunctions as in \eqref{FK}.  Thus, $(a^\pm_2 , a^\pm_0 )$ solve the following ODE system, with $\mu^{(1)}_\lambda$ playing the role of a Lagrange multiplier,
\begin{equation}\label{a0a2sys}
\left\{
\begin{array}{l}
 -(a^{\pm}_{0})''-{1\over r}(a^{\pm}_{0})'+\lambda [A_{\pm}(f^2_{\pm}-t^2_{\pm})+B(f^2_{\mp}-t^2_{\mp})]a^\pm_{0}\\
\qquad\qquad\qquad +\lambda A_{\pm}f^2_{\pm}(a^{\pm}_{0}-a^{\pm}_{2})+\lambda Bf_{+}f_{-}(a^{\mp}_{0}-a^{\mp}_{2})=\mu^{(1)}_{\lambda}a^{\pm}_{0},
\ \ \ \text{in}\ [0,1],\\
-(a^{\pm}_{2})''-{1\over r}(a^{\pm}_{2})'+{4\over r^{2}}a^{\pm}_{2}+\lambda [A_{\pm}(f^2_{\pm}-t^2_{\pm})+B(f^2_{\mp}-t^2_{\mp})]a^\pm_{2}\\
\qquad\qquad\qquad+\lambda A_{\pm}f^2_{\pm}(a^{\pm}_{2}-a^{\pm}_{0})+\lambda Bf_{+}f_{-}(a^{\mp}_{2}-a^{\mp}_{0})=\mu^{(1)}_{\lambda}a^{\pm}_{2},\ \ \ \text{in}\ [0,1],\\
a^{\pm}_{0}(1)=a^{\pm}_{2}(1)=0.
\end{array}
\right.
\end{equation}
Writing \eqref{a0a2sys} in terms of $K_{\pm}$ and $F_{\pm}$, we have:
{\allowdisplaybreaks
\begin{equation}\label{KFsys}
\left.
\begin{array}{l}
\dis -F''_{\pm}-{1\over r}F'_{\pm}+{2\over r^{2}}(F_{\pm}-K_{\pm})
+\lambda [A_{\pm}(f^2_{\pm}-t^2_{\pm})+B(f^2_{\mp}-t^2_{\mp})]F_{\pm}=\mu^{(1)}_{\lambda}F_\pm ,\ \ \ \text{in}\ \ [0, 1],\\
\dis -K''_{\pm}-{1\over r}K'_{\pm}+{2\over r^{2}}(K_{\pm}-F_{\pm})+\lambda [A_{\pm}(f^2_{\pm}-t^2_{\pm})+B(f^2_{\mp}-t^2_{\mp})]K_{\pm}\\
\quad\quad\qquad\qquad\qquad\qquad\qquad\qquad+2\lambda A_{\pm}f^{2}_{\pm}K_{\pm}+2\lambda Bf_{+}f_{-}K_{\mp}=\mu^{(1)}_{\lambda}K_{\pm}, \ \ \ \text{in}\ \ [0, 1],\\
K_{\pm}(1)=0=F_{\pm}(1),
\end{array}
\right\}
\end{equation}}
with $F_{\pm}\ge0$, $K_{\pm}\ge0$.

Now, let  $\widetilde{F}_{\pm}={f_{\pm}\over r}$ and $\widetilde{K}_{\pm}=f'_{\pm}$.  A straightforward computation shows that $\widetilde{F}_\pm, \widetilde{K}_\pm$ solve the same system of ODE but without the Lagrange multiplier  $\mu^{(1)}_{\lambda},$ that is:
\begin{equation}\label{tilde}
\left.
\begin{array}{l}
\dis -\widetilde{F}''_{\pm}-{1\over r}\widetilde{F}'_{\pm}+{2\over r^{2}}(\widetilde{F}_{\pm}-\widetilde{K}_{\pm})+\lambda [A_{\pm}(f^2_{\pm}-t^2_{\pm})+B(f^2_{\mp}-t^2_{\mp})]\widetilde{F}_{\pm}=0,\ \ \ \text{in}\ \ [0, 1],\\
\dis -\widetilde{K}''_{\pm}-{1\over r}\widetilde{K}'_{\pm}+{2\over r^{2}}(\widetilde{K}_{\pm}-\widetilde{F}_{\pm})+\lambda [A_{\pm}(f^2_{\pm}-t^2_{\pm})+B(f^2_{\mp}-t^2_{\mp})]\widetilde{K}_{\pm}\\
\quad\qquad\qquad\qquad\qquad\qquad\qquad\qquad+2\lambda A_{\pm}f^{2}_{\pm}\widetilde{K}_{\pm}+2\lambda Bf_{+}f_{-}\widetilde{K}_{\mp}=0,\ \ \ \text{in}\ \ [0, 1],\\
\dis \widetilde{K}_{\pm}(1)=f'_{\pm}(1)>0,  \ \ \ \ \widetilde{F}_{\pm}(1)=t_{\pm}.
\end{array}
\right\}
\end{equation}

We now multiply the equations in \eqref{KFsys} by $\widetilde{K}_{\pm}r$ respectively, and integrate by parts, to obtain that
\begin{multline}\label{Kinteqns}
\dis -K'_{\pm}(1)\widetilde{K}_{\pm}(1)+\int^1_{0}\widetilde{K}'_{\pm}K'_{\pm}rdr+\int^1_{0}{2\over r^{2}}(K_{\pm}-F_{\pm})\widetilde{K}_{\pm}rdr\\
\dis +\int^1_{0}\lambda [A_{\pm}(f^2_{\pm}-t^2_{\pm})+B(f^2_{\mp}-t^2_{\mp})] K_{\pm}\widetilde{K}_{\pm}rdr\\
+2\lambda\int^1_{0}[A_{\pm}f^2_{\pm}+Bf_{+}f_{-}]K_{\pm}\widetilde{K}_{\pm} rdr=\mu^{(1)}_{\lambda}\int^1_{0}\widetilde{K}_{\mp}K_{\pm}rdr.
\end{multline}

Similarly, we multiply the equations in \eqref{tilde} by $K_{\pm}r$ and integrate by parts to arrive at
\begin{multline}\label{Ktildeqns}
\dis \int^1_{0}\widetilde{K}'_{\pm}K'_{\pm}rdr+\int^1_{0}{2\over r^{2}}(\widetilde{K}_{\pm}-\widetilde{F}_{\pm})K_{\pm}rdr\\
\dis +\int^1_{0}\lambda [A_{\pm}(f^2_{\pm}-t^2_{\pm})+B(f^2_{\mp}-t^2_{\mp})] K_{\pm}\widetilde{K}_{\pm}rdr\\
+2\lambda\int^1_{0}[A_{\pm}f^2_{\pm}+Bf_{+}f_{-}]K_{\pm}\widetilde{K}_{\pm} rdr=0.
\end{multline}

Therefore, by \eqref{Kinteqns}, \eqref{Ktildeqns} and the boundary condition $K_{\pm}(1)=0$, we conclude:
\begin{equation}\label{Km2eqns}
\dis -K'_{\pm}(1)\widetilde{K}_{\pm}(1)+\int^1_{0}{2\over r^{2}}(\widetilde{F}_{\pm}K_{\pm}-\widetilde{K}_{\pm}F_{\pm})rdr
=\mu^{(1)}_{\lambda}\int^1_{0}K_{\pm}\widetilde{K}_{\pm}rdr.
\end{equation}
Similarly, we have
\begin{equation}\label{Fm2eqns}
\dis-F'_{\pm}(1)\widetilde{F}_{\pm}(1)+\int^1_{0}{2\over r^{2}}(\widetilde{K}_{\pm}F_{\pm}-\widetilde{F}_{\pm}K_{\pm})rdr
=\mu^{(1)}_{\lambda}\int^1_{0}F_{\pm}\widetilde{F}_{\pm}rdr.
\end{equation}

Combining \eqref{Km2eqns} and \eqref{Fm2eqns}, it follows that
\begin{equation}\label{m2eqn}
\mu^{(1)}_{\lambda}\int^1_{0}(\widetilde{K}_{\pm}K_{\pm}+\widetilde{F}_{\pm}F_{\pm})rdr=-K'_{\pm}(1)\widetilde{K}_{\pm}(1)-F'_{\pm}(1)\widetilde{F}_{\pm}(1).
\end{equation}
Since $K_{\pm}\ge0$, $F_{\pm}\ge0$ in $[0,1]$ and $K_{\pm}(1)=F_{\pm}(1)=0$, we have $K'_{\pm}(1)\le0$, $F'_{\pm}(1)\le0$. We claim that $K'_{\pm}(1)<0$, $F'_{\pm}(1)<0$. Indeed, if $K'_{\pm}(1)=0=F'_{\pm}(1)$, and with the boundary conditions $K_{\pm}(1)=0=F_{\pm}(1)$, it implies that zero is the only solution of \eqref{KFsys} by ODE uniqueness, i.e. $K_{\pm}(r)\equiv0\equiv F_{\pm}(r)$. Hence, $a^{\pm}_{0}(r)\equiv0\equiv a^{\pm}_{2}(r)$, which is impossible. Therefore, $K'_{\pm}(1)<0$, $F'_{\pm}(1)<0$. Since the right side of \eqref{m2eqn} is positive, and the each term on the left side of \eqref{m2eqn} is also positive except $\mu^{(1)}_{\lambda}$. Hence, we can obtain that $\mu^{(1)}_{\lambda}>0$. This completes the proof.
\end{proof}

\begin{proof}[Proof of (i) of Theorem~\ref{main}]
By the fact $\mu_{\lambda}=\min\{ \mu^{(0)}_\lambda , \mu^{(1)}_\lambda \}$, and together with Proposition~\ref{muequal}-\ref{mu0} and Lemma~\ref{mu1Bneg}, we have $\tilde{\mu}_\lambda >0$, which gives us that $E''_{\lambda}(\Psi)[\Phi]>0$. This completes our proof.
\end{proof}

\section{Instability}

We next turn to part (ii) of Theorem~\ref{main}, namely the loss of stability of the equivariant solution to \eqref{dirichlet} when $B>0$ as $\lambda$ increases.  From Propositions~\ref{muequal} and \ref{mu0}, which apply for any $B$ regardless of sign, we know that the loss of stability will occur when an eigenvalue of $\mathcal{L}^{(1)}_\lambda$ crosses zero.  The following lemma relates the eigenvalues of the operators $\mathcal{L}^{(1)}_\lambda$  and $\mathcal{M}^{(1)}_{\lambda}$ (defined in Section~3):

\begin{lemma}\label{eigenspace}
$\mu \in \mathbb{R}$ is an eigenvalue of $\mathcal{L}^{(1)}_\lambda$ over $L^{2}(([0, 1); rdr); \mathbb{C}^{4})$ if only if it is an eigenvalue of $\mathcal{M}^{(1)}_{\lambda}$ over $L^{2}(([0, 1); rdr); \mathbb{R}^{4})$. Moreover, if $\mu$ is a simple eigenvalue of $\mathcal{M}^{(1)}_{\lambda}$ with eigenspace spanned by $(a^\pm_0, a^\pm_2)$, then
$$
\mathrm{ker}(\mathcal{L}^{(1)}_\lambda -\mu I)=\left\{\ t(\xi a^\pm_0, -\overline{\xi}a^\pm_2):\ \xi\in \mathbb{S}^{1}, \ \ t\in\mathbb{R}\ \right\}.
$$
\end{lemma}
\begin{proof}  The proof follows along the lines of Lemma~5.6 of \cite{abm13}; we provide details here for completeness.
Let $\mu\in\sigma (\mathcal{L}^{(1)}_{\lambda})$ with complex-valued eigenvectors $(b^\pm_0 , b^\pm_2 )$, that is 
$$
\mathcal{L}^{(1)}_{\lambda}\left[\begin{array}{c}b^\pm_0 \\b^\pm_2\end{array}\right]=\mu \left[\begin{array}{c}b^\pm_0 \\b^\pm_2\end{array}\right] ,
$$
with $\mathcal{L}^{(1)}_{\lambda}$ defined in \eqref{L1complex}. We observe that $a^\pm_0 =\text{Im}b^\pm_0$, $a^\pm_2 =\text{Im}b^\pm_2$ will be eigenvectors of 
$\mathcal{M}^{(1)}_\lambda$ with $\mu$. On the other hand, if $(a^\pm_0 , a^\pm_2 )$ are real-valued eigenvectors of $\mathcal{M}^{(1)}_\lambda$ with $\mu$, then
$(b^\pm_0 , b^\pm_2 )=(ia^\pm_0 , ia^\pm_2 )$ will be eigenvectors of $\mathcal{L}^{(1)}_\lambda$ with the same eigenvalue. Then, 
$\sigma (\mathcal{L}^{(1)}_{\lambda})=\sigma ( \mathcal{M}^{(1)}_\lambda )$.

Now suppose $\mu$ is simple eigenvalue of $\mathcal{M}^{(1)}_\lambda$ with eigenspace spanned by $(a^\pm_0, a^\pm_2)$. If $(b^\pm_0, b^\pm_2)$ is an eigenfunction
of $\mathcal{L}^{(1)}_{\lambda}$, then (by the observation above) $(\text{Im}b^\pm_0 , \text{Im}b^\pm_2 )=l(a^\pm_0, a^\pm_2)$ for $l\in\mathbb{R}$. Similarly, 
$(\text{Re}b^\pm_0 , -\text{Re}b^\pm_2 )$ is an eigenfunction of $\mathcal{M}^{(1)}_\lambda$, and so  
$(\text{Re}b^\pm_0 , -\text{Re}b^\pm_2 )=k(a^\pm_0, a^\pm_2)$ for $k\in\mathbb{R}$. Setting $t=\sqrt{k^{2}+l^{2}}$ and $\xi={{k+il}\over t}\in \mathbb{S}^{1}$, 
we have $(b^\pm_0, b^\pm_2)=t(\xi a^\pm_0, -\overline{\xi}a^\pm_2)$ as desired.
\end{proof}

\begin{remark}\label{S1action}\rm
We observe that the Dirichlet problem \eqref{dirichlet} is invariant under an $\mathbb S^1$ group action in the sense that if $\Psi$ is a solution, then so is 
\begin{equation}\label{action}
 R_\xi\Psi(x):= \bar\xi \Psi(\xi x),  
 \end{equation}
 where $\xi=e^{i\alpha}\in\mathbb S^1\subset\CC$, and where $\xi x$ is interpreted as complex multiplication of $\xi$ with $x=x_1+ix_2\in \CC\simeq\RR^2$.  That is, a rotation of the independent variable $x$ is equivalent to the same rotation in the image of $\Psi$.  As a consequence, any eigenfunction of the linearization generates a circle of eigenfunctions (with the same eigenvalue) via this group action.  This is expressed in Lemma~\ref{eigenspace}:  the real-valued eigenfunctions $(a_0^\pm, a_2^\pm)$ of $\mathcal{M}^{(1)}_\lambda$ generate a circle of eigenfunctions of the complex operator  $\mathcal{L}^{(1)}_{\lambda}$
(and hence of the full linearized operator $\mathcal{L}_\lambda$.)
\end{remark}

As is \cite{abm13}, the ground-state eigenvalue $\mu^{(1)}_\lambda$ of $\mathcal M^{(1)}_\lambda$ is indeed a simple eigenvalue:
\begin{prop}\label{simple}
Let $\mu^{(1)}_\lambda=\inf\sigma(\mathcal M^{(1)}_\lambda)$.  Then 
$\mu^{(1)}_\lambda$ is a simple eigenvalue, and the eigenfunctions $(a_0^\pm, a_2^\pm)$ may be chosen with $0\le a_2(r)\le a_0(r)$ for all $r$.
\end{prop}

\begin{proof}
We recall that $\mathcal M^{(1)}_\lambda$ is associated to the quadratic form $Q^{(1)}_\lambda$, defined in \eqref{Q(1)}.  As $\mathcal M^{(1)}_\lambda$ is self-adjoint and has compact resolvent, the spectrum is discrete and the lowest eigenvalue $\mu^{(1)}_\lambda$ is obtained by minimizing the Raleigh quotient,
$\mu^{(1)}_\lambda=\min_{\|(a_0^\pm,a_2^\pm)\|_{L^2}=1} 
        Q^{(1)}_\lambda (a_0^\pm, a_2^\pm)$.

First, we claim that the eigenfunctions $(a_0^\pm, a_2^\pm)$ are each of fixed sign, and $|a_2^\pm(r)|\le |a_0^\pm(r)|$ holds for all $r$.  Indeed, define 
$$\tilde a_2^\pm(r)=\min\{|a_2^\pm(r)|,|a_0^\pm(r)|\}\quad\text{and}\quad \tilde a_0^\pm(r)=\max\{|a_2^\pm(r)|,|a_0^\pm(r)|\}.$$
Then it is easy to verify that 
$\|a_0^\pm\|_{L^2}^2 + \|a_2^\pm\|_{L^2}^2 = \|\tilde a_0^\pm\|_{L^2}^2 + \|\tilde a_2^\pm\|_{L^2}^2$, and 
$Q_\lambda^{(1)}(\tilde a^\pm_0,\tilde a^\pm_2) \le Q_\lambda^{(1)}(a^\pm_0,a^\pm_2)$, with strict inequality if the claim were false.  But then this would contradict the fact that $(a_0^\pm, a_2^\pm)$ minimize the Rayleigh quotient, and thus the claim is established.

The simplicity of the eigenvalue $\mu^{(1)}_\lambda$ now follows by the usual argument: eigenfunctions of fixed sign can only generate a one-dimensional eigenspace.
\end{proof}

In order to study the dependence on $\lambda$ of the eigenvalues of the linearized operator $\mathcal{M}^{(1)}_\lambda (a^\pm_0 , a^\pm_2 )$ (defined in \eqref{L1real}), it will be convenient to change variables so that the equation is fixed, but the domain grows with $\lambda$, approaching an entire solution (on all $\RR^2$) as $\lambda\to\infty$.  To this end, we define $R=\sqrt{\lambda}$, and define $\hat f_\pm(r;R) = f_\pm(rR;\lambda)$, $\hat{a}^\pm_0 (r;R) =a^\pm_0 (rR;\lambda)$, $\hat{a}^\pm_2 (r;R) =a^\pm_2 (rR;\lambda)$.  The profile functions $\hat f_\pm$
will solve \eqref{radeq}, and by this rescaling we represent quadratic form
\begin{align*}
&Q^{(1)}_{\lambda}(a^\pm_0, a^\pm_2 )=2\pi\int^R_0 \sum_{i=\pm}\left[ |(\hat{a}^i_2 )'|^{2}+ |(\hat{a}^i_0 )'|^{2}+{4\over r^2 } |\hat{a}^i_2 |^{2}\right]rdr\\
&\quad\quad+2\pi\int^R_0 \left[ A_{+}f^2_{+}(\hat{a}^+_0 -\hat{a}^+_2 )^{2}+A_{-}f^2_{-}(\hat{a}^-_0 -\hat{a}^-_2 )^{2}
+2Bf_{+}f_{-}(\hat{a}^+_0 -\hat{a}^+_2 )(\hat{a}^-_0 -\hat{a}^-_2 )\right]rdr\\
&\quad\quad+2\pi\int^R_0 \left[ A_{+}(f^2_{+}-t^2_{+})+B(f^2_{-}-t^2_{-})\right](|\hat{a}^+_0 |^{2}+|\hat{a}^+_2 |^{2})rdr\\
&\quad\quad+2\pi\int^R_0 \left[ A_{-}(f^2_{-}-t^2_{-})+B(f^2_{+}-t^2_{+})\right](|\hat{a}^-_0 |^{2}+|\hat{a}^-_2 |^{2})rdr\\
&\qquad=:\widehat{Q}_{R}(\hat{a}^\pm_0, \hat{a}^\pm_2 ) ,
\end{align*}
in terms of the rescaled functions $(\hat{a}^\pm_0, \hat{a}^\pm_2 )$ on $r\in (0,R)$.
The associated self-adjoint operator to $\widehat{Q}_{R}(\hat{a}^\pm_0, \hat{a}^\pm_2 )$ is defined as $\widehat{\mathcal{M}}^{(1)}_R$. We observe that the first eigenvalue of $\widehat{\mathcal{M}}^{(1)}_R$ denoted by $\hat{\mu}_{R}$, is related to the first eigenvalue $\mu^{(1)}_{\lambda}$ via $\mu^{(1)}_{\lambda}=R^{2}\hat{\mu}_{R}$, and so the ground state eigenvalues of $\mathcal{M}_\lambda^{(1)}$ and 
$\widehat{\mathcal{M}}^{(1)}_R$ have the same sign under this transformation.
 
With the aid of this rescaling we obtain the desired instability result for equivariant solutions when $B>0$:

 \begin{theorem}\label{instability}
For any $B\in (0, B_0 )$ with $B_0$ as same as in Theorem~\ref{monotone}, there exists a constant $R_{*}=R_{*}(B)>0$ such that $\hat\mu_R<0$ for any $R>R_{*}$.
 \end{theorem}
\begin{proof}
To do this we argue as in Theorem 2 in \cite{m95} (see also Lemma~5.21 of \cite{abm13}). For any admissible $\hat{a}^\pm_0 (r)$, $\hat{a}^\pm_2 (r)$, define 
$L_{\pm}=\hat{a}^\pm_0 (r)+\hat{a}^\pm_2 (r)$, $P_{\pm}=\hat{a}^\pm_0 (r)-\hat{a}^\pm_2 (r)$. 
In other words, we take $\hat{a}^\pm_0, \hat{a}^\pm_2$ of the form,
$$  \hat{a}^\pm_0=\frac12\left( L_\pm + P_\pm\right), \qquad
    \hat{a}^\pm_2=\frac12\left( L_\pm - P_\pm\right).
$$
We can rewrite $\widehat{Q}_{R}(\hat{a}^\pm_0, \hat{a}^\pm_2 ) $
in terms of $L_\pm, P_\pm$ as follows:
\begin{align*}
\widehat{Q}_{R}(\hat{a}^\pm_0, \hat{a}^\pm_2 )&=\pi\int^R_0 \sum_{i=\pm}\left[ |P'_i |^{2}+ |L'_i |^{2}+{2\over r^2 } |L_i -P_i  |^{2}\right]rdr\\
&\quad+2\pi\int^R_0 \left[ A_{+}f^2_{+}P^{2}_{+}+A_{-}f^2_{-}P^{2}_{-}
+2Bf_{+}f_{-}P_{+}P_{-}\right]rdr\\
&\quad+\pi\int^R_0 \left[ A_{+}(f^2_{+}-t^2_{+})+B(f^2_{-}-t^2_{-})\right](|P_+ |^{2}+|L_+ |^{2})rdr\\
&\quad+\pi\int^R_0 \left[ A_{-}(f^2_{-}-t^2_{-})+B(f^2_{+}-t^2_{+})\right](|P_- |^{2}+|L_- |^{2})rdr\\
&=: \breve{Q}_{R}(L_\pm, P_\pm ) .
\end{align*}

We now choose $L_\pm,P_\pm$ to make the quadratic form negative for all large $R$, using the remark \eqref{split} from the Introduction.
As $R\to\infty$, the radial profile $f_{\pm}(\cdot, R)\to f^\infty_\pm (\cdot)$ in $C^{k}([0, R])$ for all $R>0$ and $k\in\mathbb{N}$, with $f^\infty_\pm$ the modulus of the unique
entire equivariant solution of the form $\psi^\infty_\pm =f^\infty_\pm (r)e^{i\theta}$. Let $L_+ ={f^\infty_{+}\over r}$, $L_- =-{f^\infty_{-}\over r}$, $P_+ =(f^\infty_+ )'$, $P_- =-(f^\infty_- )'$. 
Since $f^\infty_\pm$ vanish linearly at $r=0$, $L_{\pm}(r)$ and $P_{\pm}(r)$ are regular near $r=0$,
$L_{+}-P_{+}=-r\left[ {f^\infty_{+}\over r}\right]'$, $L_{-}-P_{-}=r\left[ {f^\infty_{-}\over r}\right]'$ are well-defined in $\breve{Q}_{R}(L_\pm, P_\pm )$. Meanwhile, $P_\pm$ and $L_\pm$
satisfy the following equations:
\begin{equation*}
\left\{
\begin{aligned}
0&= -(L_\pm )''-{1\over r}(L_\pm )'+{2\over r}(L_\pm -P_\pm )+[A_{\pm}((f^\infty_{\pm})^{2}-t^2_\pm)+B((f^\infty_{\mp})^{2}-t^2_{\mp})]L_{\pm}, \\
0&= -(P_\pm )''-{1\over r}(P_\pm )'+{2\over r}(P_\pm -L_\pm )+[A_{\pm}((f^\infty_{\pm})^{2}-t^2_\pm)+B((f^\infty_{\mp})^{2}-t^2_{\mp})]P_{\pm}\\
&\qquad\qquad\qquad\qquad\qquad\qquad\qquad\qquad+2A_{\pm}(f^\infty_{\pm})^{2}P_{\pm}-2Bf^\infty_{+}f^\infty_{-}P_{\mp} .
\end{aligned}\right.
\end{equation*}
Using above equations and integrating by parts, together with the asymptotic properties of radial solutions at infinity in Theorem~\ref{asymptotics}, we can obtain that
{\allowdisplaybreaks
\begin{align*}
&\lim_{R\to\infty}\breve{Q}_{R}(L_\pm, P_\pm )\\
&\quad\quad=\lim_{R\to\infty}\pi\left[ P'_{\pm}P_{\pm}r\big|^R_{0}+L'_{\pm}L_{\pm}r\big|^R_{0}\right]\\
&\quad\qquad+\pi\int^\infty_{0}\left\{-{2\over r}(P_\pm -L_\pm )P_\pm -[A_{\pm}((f^\infty_{\pm})^{2}-t^2_\pm)+B((f^\infty_{\mp})^{2}-t^2_{\mp})]P^2_{\pm}\right.\\
&\quad\qquad\qquad\qquad\left.-2A_{\pm}(f^\infty_{\pm})^{2}P^2_{\pm}+4Bf^\infty_{+}f^\infty_{-}P_{+}P_{-} \right\}rdr\\
&\quad\qquad+\pi\int^\infty_{0}\left\{-{2\over r}(L_\pm -P_\pm )L_\pm -[A_{\pm}((f^\infty_{\pm})^{2}-t^2_\pm)+B((f^\infty_{\mp})^{2}-t^2_{\mp})]L^2_{\pm}\right\}rdr\\
&\quad\qquad+\pi\int^\infty_{0} {2\over r^2 }|L_\pm -P_\pm |^{2}rdr\\
&\quad\qquad+2\pi\int^\infty_0 \left[ A_{+}(f^\infty_{+})^{2}P^{2}_{+}+A_{-}(f^\infty_{-})^{2}P^{2}_{-}+2Bf^\infty_{+}f^\infty_{-}P_{+}P_{-}\right]rdr\\
&\quad\qquad+\pi\int^\infty_0 \left[ A_{+}((f^\infty_{+})^{2}-t^2_{+})+B((f^\infty_{-})^{2}-t^2_{-})\right](|P_+ |^{2}+|L_+ |^{2})rdr\\
&\quad\qquad+\pi\int^\infty_0 \left[ A_{-}((f^\infty_{-})^{2}-t^2_{-})+B((f^\infty_{+})^{2}-t^2_{+})\right](|P_- |^{2}+|L_- |^{2})rdr\\
&\quad\quad=8\pi B\int^\infty_0 f^\infty_{+}f^\infty_{-}P_{+}P_{-}rdr\\
&\quad\quad=-8\pi B\int^\infty_0 f^\infty_{+}f^\infty_{-}(f^\infty_{+})'(f^\infty_{-})'rdr<0 ,
\end{align*}}
since by part (ii) of Theorem~\ref{monotone} we have $(f^{\infty}_\pm )'(r)>0$ when $0<B<B_0$. 

Denote by 
$\breve{Q}_{\infty}(\breve{L}_{\pm}, \breve{P}_{\pm}):=\dis\lim_{R\to\infty}\breve{Q}_{R}(L^R_\pm, P^R_\pm )<0$, by the above computation. 
Set a cut-off function $\eta_{R}(x)=\eta_{1}(x/R)$ with $\eta_{1}(x)=1$ for $0\le x\le 1/2$, $\eta_{1}(x)=0$ for $x>1$ and $0<\eta_{1}(x)\le1$ for $1/2 <x<1$. Define $L^{R}_{\pm}:=\breve{L}_{\pm}\eta_R$, $P^R_\pm :=\breve{P}_{\pm}\eta_R \in H^1_{0}([0, R))$.  By the asymptotic expansions stated in Theorem~\ref{asymptotics}, it follows that
$$  \breve{Q}_{\infty}(\breve{L}_{\pm}, \breve{P}_{\pm})
= \lim_{R\to\infty} 
    \breve{Q}_{\infty}(\breve{L}^R_{\pm}, \breve{P}^R_{\pm}), $$
and so we may find $R_*=R_*(B)$ sufficiently large that
$\breve{Q}_{\infty}(\breve{L}^R_{\pm}, \breve{P}^R_{\pm})<0$ for all $R>R_*$.
 Thus, for $R>R_*(B)$ and $0<B<B_0$, we have shown that
$\mu^{(1)}_{1}(R^2 )=R^{2}\hat{\mu}_{1}(R)<0$. Since $\mu_\lambda =\min\{\mu^{(0)}_\lambda, \mu^{(1)}_\lambda \}<\mu^{(1)}_\lambda$, we have $\mu_\lambda <0$ for all $\lambda>\lambda_*:=R^2_*$, which completes our proof. 
\end{proof}

The proof of part (ii) of Theorem~\ref{main} then follows from Theorem~\ref{instability} and the fact that $\mu_\lambda^{(1)}$ is an eigenvalue of $\mathcal{L}_\lambda$, and $\mu_\lambda^{(1)}=R^2\hat\mu_R<0$ when $\lambda>R^2$ and $R>R_*$.

As a final remark, we note that, by Proposition~\ref{muequal}, Proposition~\ref{mu0},  Lemma~\ref{eigenspace}, and Proposition~\ref{simple}, the eigenfunction of $\mathcal{L}_\lambda$ which is associated to the negative ground state has the form
\begin{align*}
   \Phi &= \left( \xi a_0^+(r) - \bar\xi a_2^+(r)e^{2i\theta}
                \, , \, 
                           \xi a_0^-(r) - \bar\xi a_2^-(r)e^{2i\theta}\right)\\
        & \simeq {1\over 2} 
         R_\xi\left( \left[(f_+^\infty)' + {f_+^\infty\over r}\right] +      
             \left[(f_+^\infty)' - {f_+^\infty\over r}\right]e^{2i\theta}             
             \, , \, 
                 - \left[(f_-^\infty)' + {f_-^\infty\over r}\right]
                 -\left[(f_-^\infty)' - {f_-^\infty\over r}\right]e^{2i\theta}\right)
                 \\
          & = R_\xi \left(\partial_{x_1} \left[f^\infty_+(r)e^{i\theta}\right]
             \, , \, 
               -\partial_{x_1} \left[f^\infty_-(r)e^{i\theta}\right]\right),
\end{align*}
as $R\to\infty$, (with $R_\xi$, $\xi\in \mathbb S^1$ as in \eqref{action}) as promised in \eqref{split}.  In other words, the energy decreases along a direction which (roughly) corresponds to antipodal motion of the vortices in the respective components.  As in \cite{abm13}, this suggests a symmetry-breaking bifurcation at a critical value $\lambda_*$, at which the equivariant solution loses stability and gives rise to a non-symmetric family of solutions with distinct vortices in each component $\psi_\pm$. We conjecture that this bifurcation does indeed occur, but the analytic details of such a result remain an open problem.

\bibliography{myrefs}{}

\newcommand{\etalchar}[1]{$^{#1}$}
\begin{thebibliography}{EKN{\etalchar{+}}11}

\bibitem[ABM09]{abm09}
Stan Alama, Lia Bronsard, and Petru Mironescu.
\newblock On the structure of fractional degree vortices in a spinor
  {G}inzburg-{L}andau model.
\newblock {\em J. Funct. Anal.}, 256(4):1118--1136, 2009.

\bibitem[ABM12]{abm13}
Stan Alama, Lia Bronsard, and Petru Mironescu.
\newblock On compound vortices in a two-component {G}inzburg-{L}andau
  functional.
\newblock {\em {\rm Preprint arXiv.org:1211.5657v1. To appear in} Indiana Univ.
  Math. Journal.}, 2012.

\bibitem[AG12]{ag13}
Stan Alama and Qi~Gao.
\newblock Symmetric vortices for two-component {G}inzburg--{L}andau systems.
\newblock {\em {\rm Preprint arXiv.org:1211.5652v1. To appear in} J.
  Differential Equations}, 2012.

\bibitem[BBH94]{bbh94book}
Fabrice Bethuel, Ha{\"{\i}}m Brezis, and Fr{\'e}d{\'e}ric H{\'e}lein.
\newblock {\em Ginzburg-{L}andau vortices}.
\newblock Progress in Nonlinear Differential Equations and their Applications,
  13. Birkh\"auser Boston Inc., Boston, MA, 1994.

\bibitem[BCP93]{bcp93}
Patricia Bauman, Neil~N. Carlson, and Daniel Phillips.
\newblock On the zeros of solutions to {G}inzburg-{L}andau type systems.
\newblock {\em SIAM J. Math. Anal.}, 24(5):1283--1293, 1993.

\bibitem[BO86]{bo86}
Ha{\"{\i}}m Brezis and Luc Oswald.
\newblock Remarks on sublinear elliptic equations.
\newblock {\em Nonlinear Anal.}, 10(1):55--64, 1986.

\bibitem[Bre11]{b11}
Haim Brezis.
\newblock {\em Functional analysis, {S}obolev spaces and partial differential
  equations}.
\newblock Universitext. Springer, New York, 2011.

\bibitem[CM98]{cm98}
Myriam Comte and Petru Mironescu.
\newblock A bifurcation analysis for the {G}inzburg-{L}andau equation.
\newblock {\em Arch. Rational Mech. Anal.}, 144(4):301--311, 1998.

\bibitem[EKN{\etalchar{+}}11]{eto11}
Minoru Eto, Kenichi Kasamatsu, Muneto Nitta, Hiromitsu Takeuchi, and Makoto
  Tsubota.
\newblock Interaction of half-quantized vortices in two-component bose-einstein
  condensates.
\newblock {\em Phys. Rev. A}, 83:063603, Jun 2011.

\bibitem[Gao13]{thesis}
Qi~Gao.
\newblock {\em {Solutions of a Two-Component Ginzburg--Landau System}}.
\newblock PhD thesis, McMaster University, 2013.

\bibitem[GT01]{gt01}
David Gilbarg and Neil~S. Trudinger.
\newblock {\em Elliptic partial differential equations of second order}.
\newblock Classics in Mathematics. Springer-Verlag, Berlin, 2001.
\newblock Reprint of the 1998 edition.

\bibitem[HH94]{hh94}
Rose-Marie Herv{\'e} and Michel Herv{\'e}.
\newblock \'{E}tude qualitative des solutions r\'eelles d'une \'equation
  diff\'erentielle li\'ee \`a l'\'equation de {G}inzburg-{L}andau.
\newblock {\em Ann. Inst. H. Poincar\'e Anal. Non Lin\'eaire}, 11(4):427--440,
  1994.

\bibitem[KR98]{kr98}
A.~Knigavko and B.~Rosenstein.
\newblock Spontaneous vortex state and ferromagnetic behavior of type-ii p-wave
  superconductors.
\newblock {\em Phys. Rev. B}, 58:9354--9364, Oct 1998.

\bibitem[KTU03]{ktu03}
Kenichi Kasamatsu, Makoto Tsubota, and Masahito Ueda.
\newblock Structure of vortex lattices in rotating two-component bose-einstein
  condensates.
\newblock {\em Physica B: Condensed Matter}, 329-333, Part 1(0):23 -- 24, 2003.
\newblock Proceedings of the 23rd International Conference on Low Temperature
  Physics.

\bibitem[Mir95]{m95}
Petru Mironescu.
\newblock On the stability of radial solutions of the {G}inzburg-{L}andau
  equation.
\newblock {\em J. Funct. Anal.}, 130(2):334--344, 1995.

\bibitem[Pel11]{pe11}
Dmitry~E. Pelinovsky.
\newblock {\em Localization in periodic potentials}, volume 390 of {\em London
  Mathematical Society Lecture Note Series}.
\newblock Cambridge University Press, Cambridge, 2011.
\newblock From Schr{\"o}dinger operators to the Gross-Pitaevskii equation.

\bibitem[Sau03]{sauv03}
Myrto Sauvageot.
\newblock Properties of the solutions of the ginzburg-landau equation on the
  bifurcation branch.
\newblock {\em NoDEA Nonlinear Differential Equations Appl.}, 10(4):375--397,
  2003.

\end{thebibliography}
\bibliographystyle{alpha}

\nocite{gt01} 
\nocite{b11}

\end{document}